\documentclass[12pt,a4paper]{article}
\usepackage{amssymb}
\usepackage{amsmath}
\usepackage[symbol]{footmisc}
\usepackage{amsthm}
\usepackage{amscd}
\usepackage{multirow}
\usepackage{amstext}
\usepackage{booktabs}
\usepackage{vmargin}
\usepackage{float}
\usepackage{fancyhdr}
\setpapersize{A4}
\setmarginsrb{3.0cm}{1.5cm}{3.0cm}{3cm}{1cm}{.5cm}{2cm}{2cm}
\usepackage{amsfonts}
\usepackage{amsthm}
\usepackage[numbers]{natbib}
\usepackage[latin1]{inputenc}
\usepackage[T1]{fontenc}
\usepackage{enumerate}
\usepackage[american]{babel}
\usepackage{graphicx}
\usepackage{bbm}
\usepackage[mathscr]{euscript}
\usepackage{mathrsfs}
\usepackage{stmaryrd}
\usepackage{soul}
\usepackage{hyperref}
\usepackage{xspace}
\usepackage[draft,danish]{fixme}
\usepackage{mathtools}
\usepackage{dsfont}
\usepackage{url}
\usepackage[ps,all,dvips]{xy}
\SelectTips{cm}{12}
\CompileMatrices

  \makeatletter
 \useshorthands{"}
 \defineshorthand{"-}{\nobreak-\bbl@allowhyphens}
 \makeatother

\theoremstyle{plain}
\newtheorem{lemma}{Lemma}[section]

\newtheorem{theorem}[lemma]{Theorem}

\theoremstyle{definition}
\newtheorem{example}[lemma]{Example}

\DeclareMathOperator*{\argmin}{arg\,min}

\theoremstyle{definition}

\theoremstyle{remark}
\newtheorem{remark}[lemma]{Remark}

\newcommand{\devnull}[1]{}

\numberwithin{equation}{section}

\title{Limit theorems for quadratic forms and related quantities of discretely sampled continuous-time moving averages}
\author{Mikkel Slot Nielsen and Jan Pedersen}
\date{May 8, 2018}
\date{\small Department of Mathematics, Aarhus University,\\
 \{mikkel, jan\}@math.au.dk}
\begin{document}
\maketitle
  %%%%%%%%%%%%%%%%%%

\begin{abstract}
The limiting behavior of Toeplitz type quadratic forms of stationary processes has received much attention through decades, particularly due to its importance in statistical estimation of the spectrum. In the present paper we study such quantities in the case where the stationary process is a discretely sampled continuous-time moving average driven by a L\'{e}vy process. We obtain sufficient conditions, in terms of the kernel of the moving and the coefficients of the quadratic form, ensuring that the centered and adequately normalized version of the quadratic form converges weakly to a Gaussian limit.
\\ \\
\footnotesize \textit{AMS 2010 subject classifications:} 60F05; 60G10; 60G51; 60H05
\\
\textit{Keywords:} Limit theorems; L\'{e}vy processes; Moving averages; Quadratic forms
\end{abstract}

%%%%%%%%%%%%%%%%%%%%%%%%%%%%%%%%%%%%%%%%%%%%%%%%%%%%%%%%%%%%%%%%%%%%%%%%%%%%%%%%%%%%%%%

\section{Introduction}\label{intro}
Let $(Y_t)_{t\in \mathbb{Z}}$ be a stationary sequence of random variables with $\mathbb{E}Y_0 = 0$ and $\mathbb{E}Y_0^2 <\infty$, and suppose that $(Y_t)_{t\in \mathbb{Z}}$ is characterized by a parameter $\theta$ which we, for simplicity, assume to be an element of $\mathbb{R}$. If one wants to infer the true value $\theta_0$ of $\theta$ from a sample $Y(n)=(Y_1,\dots, Y_n)^T$, a typical estimator is obtained as
\begin{align*}
\hat{\theta}_n = \argmin_{\theta} \ell_n (\theta),
\end{align*}
where $\ell_n = \ell_n (\cdot ; Y(n))$ is a suitable objective function. On an informal level, the usual strategy for showing asymptotic normality of the estimator $\hat{\theta}_n$ is to use a Taylor series expansion to write
\begin{align*}
\frac{\ell_n'(\theta_0)}{\sqrt{n}} = - \frac{\ell_n'' (\theta_n^*)}{n} \sqrt{n}(\hat{\theta}_n-\theta_0),
\end{align*}
and then show that $\ell_n'' (\theta_n^*)/n$ converges in probability to a non-zero constant and $\ell_n'(\theta_0)/\sqrt{n}$ converges in distribution to a centered Gaussian random variable. Here $\ell_n'$ and $\ell_n''$ refer to the first and second order derivative of $\ell_n$ with respect to $\theta$, respectively, and $\theta_n^*$ is a point in the interval formed by $\hat{\theta}_n$ and $\theta_0$. While the convergence of $\ell_n'' (\theta_n^*)/n$ usually can be shown by an ergodic theorem under the assumptions of consistency of $\hat{\theta}_n$ and ergodicity of $(Y_t)_{t\in \mathbb{Z}}$, showing the desired convergence of $\ell_n'(\theta_0)/\sqrt{n}$ may be much more challenging. In particular, if the quantity $\ell_n' (\theta_0)$ corresponds to a rather complicated function of $Y(n)$, one often needs to impose restrictive assumptions on the dependence structure of $(Y_t)_{t\in \mathbb{Z}}$, e.g., rapidly decaying mixing coefficients. In addition to the concern that such type mixing conditions do not hold in the presence of long memory, they may generally be difficult to verify.

When $\ell_n$ has an explicit form, one can sometimes exploit the particular structure to prove asymptotic normality of $\ell_n'(\theta_0)/\sqrt{n}$. To be concrete, let $\gamma_Y (\cdot ;\theta)$ denote the autocovariance function of $(Y_t)_{t\in \mathbb{Z}}$ and $\Sigma_n (\theta) = [\gamma_Y(j-k;\theta)]_{j,k=1,\dots, n}$ the covariance matrix of $Y(n)$. A very popular choice of $\ell_n$ is the (scaled) negative Gaussian log-likelihood, 
\begin{align}\label{trueLikeli}
\ell_n(\theta) = \log \det (\Sigma_n (\theta)) + Y(n)^T \Sigma_n (\theta)^{-1}Y(n).
\end{align}
In order to avoid the cumbersome and, in the presence of long memory, unstable computations related to the inversion of $\Sigma_n(\theta)$, one sometimes instead uses Whittle's approximation of (\ref{trueLikeli}), which is given by
\begin{align}\label{whittleLikeli}
\begin{aligned}
\ell_{n,\textit{Whittle}}(\theta) &= \frac{n}{2\pi}\int_{-\pi}^\pi \log (2\pi f_Y(y;\theta))\, dy + Y(n)^T A_n(\theta)Y(n)\\
&= \frac{n}{2\pi} \int_{-\pi}^\pi \Bigr[\log (2\pi f_Y(y;\theta))\, dy + \frac{I_Y(y)}{2\pi f_Y(y;\theta)}\Bigr]\, dy,
\end{aligned}
\end{align}
where $f_Y(\cdot;\theta)$ is the spectral density of $Y$, $I_Y$ is the periodogram of $Y$ and 
\begin{align*}
A_n(\theta) = \Bigr[\frac{1}{(2\pi)^2}\int_{-\pi}^\pi e^{i(j-k)y}\frac{1}{f_Y(y;\theta)}\, dy \Bigr]_{j,k=1,\dots, n}.
\end{align*}
(For details about the relation between the Gaussian likelihood and Whittle's approximation, and for some justification for their use, see \cite{beran2016long,giraitis2012large,pipiras2017long}.) An important feature of both (\ref{trueLikeli}) and (\ref{whittleLikeli}) is that, under suitable assumptions on $\gamma_Y(\cdot ;\theta)$ and $f_Y(\cdot; \theta)$, the quantities $\ell_n'( \theta_0)/\sqrt{n}$ and $\ell_{n,\textit{Whittle}}'(\theta_0)/\sqrt{n}$ are of the form $(Q_n-\mathbb{E}Q_n)/\sqrt{n}$, where
\begin{align}\label{QnQuantity}
Q_n = \sum_{t,s=1}^n b(t-s)Y_t Y_s
\end{align}
and $b:\mathbb{Z}\to \mathbb{R}$ is an even function. Consequently, proving asymptotic normality of $\ell_n'(\theta_0)/\sqrt{n}$ and $\ell_{n,\textit{Whittle}}'(\theta_0)/\sqrt{n}$ reduces to determining for which processes $(Y_t)_{t\in \mathbb{Z}}$ and functions $b$, $(Q_n-\mathbb{E}Q_n)/\sqrt{n}$ converges in distribution to a centered Gaussian random variable. In the case where $(Y_t)_{t\in \mathbb{Z}}$ is Gaussian and $b(t) = \int_{-\pi}^\pi e^{ity}\, \hat{b}(y)\, dy$, the papers \cite{avram1988bilinear,fox1987central} give conditions on $\hat{b}$ and the spectral density of $(Y_t)_{t\in \mathbb{Z}}$ ensuring that such weak convergence holds. Moreover, \citet{fox1985noncentral} proved non-central limit theorems for (an adequately normalized version of) (\ref{QnQuantity}) in case $Y_t = H(X_t)$ where $H$ is a Hermite polynomial and $(X_t)_{t\in \mathbb{Z}}$ is a normalized Gaussian sequence with a slowly decaying autocovariance function. In particular, they showed that the limit can be both Gaussian and non-Gaussian depending on the decay-rate of the autocovariances. Later, \citet{giraitis1990central} left the Gaussian framework and considered instead general linear processes of the form
\begin{align}\label{DMA}
Y_t =\sum_{s \in \mathbb{Z}} \varphi_{t-s}\varepsilon_s,\quad t \in \mathbb{Z},
\end{align}
where $(\varepsilon_t)_{t\in \mathbb{Z}}$ is an i.i.d. sequence with $\mathbb{E}\varepsilon_0 = 0$ and $\mathbb{E}\varepsilon_0^4< \infty$, and $\sum_{t \in \mathbb{Z}}\varphi^2_t<\infty$. They provided sufficient conditions (in terms of $b$ and the autocovariance function of $(Y_t)_{t\in \mathbb{Z}}$) ensuring that $(Q_n-\mathbb{E}Q_n)/\sqrt{n}$ tends to a Gaussian limit. Many interesting processes are given by (\ref{DMA}), the short-memory ARMA processes and the long-memory ARFIMA processes being the main examples, and their properties have been studied extensively. The literature on these processes is overwhelming, and the following references form only a small sample: \cite{brockDavis,long_range,giraitis2012large,hamilton1994time}.

The continuous-time analogue of (\ref{DMA}) is the moving average process $(X_t)_{t\in \mathbb{R}}$ given by
\begin{align}\label{CMA}
X_t = \int_\mathbb{R}\varphi (t-s)\, dL_s,\quad t \in \mathbb{R},
\end{align}
where $(L_t)_{t\in \mathbb{R}}$ is a two-sided L\'{e}vy process with $\mathbb{E}L_1=0$ and $\mathbb{E}L_1^4<\infty$, and where $\varphi:\mathbb{R}\to \mathbb{R}$ is a function in $L^2$. Among popular and well-studied continuous-time moving averages are the CARMA processes, particularly the Ornstein-Uhlenbeck process, and solutions to linear stochastic delay differential equations (see \cite{brockwell2001levy,brockwell2009existence,gushchin2000stationary,kuchler2013statistical} for more on these processes). \citet{bai2016limit} considered a continuous-time version of (\ref{QnQuantity}), where sums are replaced by integrals and $(Y_t)_{t\in \mathbb{Z}}$ by $(X_t)_{t\in \mathbb{R}}$ defined in (\ref{CMA}), and they obtained conditions on $b$ and $\varphi$ ensuring both a Gaussian and non-Gaussian limit for (a suitable normalized version of) the quadratic form.

Our main contribution is Theorem~\ref{theoremIntro}, which gives sufficient conditions on $\varphi$ and $b$ ensuring that $(Q_n-\mathbb{E}Q_n)/\sqrt{n}$ converges in distribution to a centered Gaussian random variable when $Y_t = X_{t\Delta}$, $t\in \mathbb{Z}$, for some fixed $\Delta>0$. In the formulation we denote by $\kappa_4$ the fourth cumulant of $L_1$ and by $\gamma_X$ the autocovariance function of $(X_t)_{t\in \mathbb{R}}$ (see the formula in (\ref{autCovQform})).
\begin{theorem}\label{theoremIntro}
Let $(X_t)_{t\in \mathbb{R}}$ be given by (\ref{CMA}) and define $Q_n$ as in (\ref{QnQuantity}) with $Y_t = X_{t\Delta}$ for some $\Delta >0$. Suppose that one of the following statements hold:
\begin{enumerate}[(i)]
\item\label{introAs1} There exist $\alpha, \beta \in [1,2]$ with $2/\alpha + 1/\beta \geq 5/2$, such that $\sum_{t\in \mathbb{Z}}\vert b(t)\vert^\beta<\infty$ and
\begin{align*}
\Bigr(t  \mapsto \sum_{s\in \mathbb{Z}} \vert \varphi(t+s\Delta)\vert^\kappa\Bigr)\in L^{4/\kappa}([0,\Delta])\quad \text{for } \kappa = \alpha,2.
\end{align*}

\item\label{introAs2} The function $\varphi$ belongs to $L^4$ and there exist $\alpha,\beta >0$ with $\alpha + \beta <1/2$, such that
\begin{align*}
\sup_{t\in \mathbb{R}} \vert t \vert^{1- \alpha/2}\vert \varphi (t)\vert< \infty \quad \text{and}
\quad \sup_{t\in \mathbb{Z}} \vert t\vert^{1-\beta} \vert b (t)\vert <\infty.
\end{align*}
\end{enumerate}
Then, as $n\to \infty$, $(Q_n - \mathbb{E}Q_n)/\sqrt{n}$ tends to a Gaussian random variable with mean zero and variance
\begin{align*}
\eta^2 =&\, \kappa_4\int_0^\Delta \Bigr(\sum_{s\in \mathbb{Z}} \varphi (t+s\Delta) \sum_{u \in \mathbb{Z}} b(u)\varphi (t+(s+u)\Delta) \Bigr)^2\, dt \\
&+2\sum_{s\in \mathbb{Z}}\Bigr(\sum_{u\in \mathbb{Z}} b(u)\gamma_X((s+u)\Delta) \Bigr)^2.
\end{align*}
\end{theorem}

While the statement in (\ref{introAs1}) is more general than the statement in (\ref{introAs2}) of Theorem~\ref{theoremIntro}, the latter provides an easy-to-check condition in terms of the decay of $\varphi$ and $b$ at infinity. Theorem~\ref{theoremIntro} relies on an approximation of $Q_n$ by a quantity of the type
\begin{align}\label{SnQuantity}
S_n = \sum_{t=1}^n X^1_{t\Delta}X^2_{t\Delta},
\end{align}
where $(X^1_t)_{t\in \mathbb{R}}$ and $(X^2_t)_{t\in \mathbb{R}}$ are moving averages of the form (\ref{CMA}), and a limit theorem for $(S_n-\mathbb{E}S_n)/\sqrt{n}$. This idea is borrowed from \cite{giraitis1990central}. Although we can use the same overall idea, $(X_{t\Delta})_{t\in \mathbb{Z}}$ is generally not of the form (\ref{DMA}) and, due to the interplay between the continuous-time specification (\ref{CMA}) and the discrete-time (low-frequency) sampling scheme, quantities such as the spectral density become less tractable. The conditions of Theorem~\ref{theoremIntro} are similar to the rather general results of \cite{bai2016limit}, which concerned the continuous-time version of (\ref{QnQuantity}). A reason that we obtain conditions of the same type as \cite{bai2016limit} is that our proofs, too, rely on (various modifications of) Young's inequality for convolutions. Since the setup of that paper requires a continuum of observations of $(X_t)_{t\in \mathbb{R}}$, those results cannot be applied in our case.

In addition to its purpose as a tool in the proof of Theorem~\ref{theoremIntro}, a limit theorem for $(S_n-\mathbb{E}S_n)/\sqrt{n}$ is of independent interest, e.g., since it is of the same form as the (scaled) sample autocovariance of (\ref{CMA}) and of $\ell_n'(\theta_0)/\sqrt{n}$ when $\ell_n$ is a suitable least squares objective function (see Example~\ref{sampleAutocovariances} and \ref{leastSquares} for details). For this reason, we present our limit theorem for $(S_n- \mathbb{E}S_n)/\sqrt{n}$ here:

\begin{theorem}\label{theoremIntro2} Let $(X_t^1)_{t\in \mathbb{R}}$ and $(X^2_t)_{t\in \mathbb{R}}$ be as in (\ref{CMA}) with corresponding kernels $\varphi_1,\varphi_2 \in L^2$ and define $S_n$ by (\ref{SnQuantity}). Suppose that one of the following statements holds:
\begin{enumerate}[(i)]
\item\label{introAs12} There exist $\alpha_1,\alpha_2\in [1,2]$ with $1/\alpha_1 + 1/\alpha_2 \geq 3/2$, such that
\begin{align*}
\Bigr(t\mapsto \sum_{s\in \mathbb{Z}}\big(\vert \varphi_i(t+s\Delta)\vert^{\alpha_i} + \varphi_i (t+s\Delta)^2 \big) \Bigr)\in L^2([0,\Delta]),\quad i =1,2.
\end{align*}
\item\label{introAs22} The functions $\varphi_1$ and $\varphi_2$ belong to $L^4$ and there exist $\alpha_1,\alpha_2 \in (1/2,1)$ with $\alpha_1+\alpha_2> 3/2$, such that
\begin{align*}
\sup_{t\in \mathbb{R}}\vert t\vert^{\alpha_i}\vert \varphi_i (t)\vert <\infty, \quad i =1,2.
\end{align*}
\end{enumerate}
Then, as $n\to \infty$, $(S_n-\mathbb{E}S_n)/\sqrt{n}$ tends to a Gaussian random variable with mean zero and variance
\begin{align*}
\eta^2 =&\, \kappa_4\int_0^\Delta \Bigr(\sum_{s\in \mathbb{Z}}\varphi_1(t+s\Delta)\varphi_2(t+s\Delta) \Bigr)^2\, dt + \mathbb{E}\big[L_1^2\big]^2\sum_{s\in \mathbb{Z}}\Bigr(\int_\mathbb{R}\varphi_1(t) \varphi_1(t+s\Delta)\, dt\\
&\cdot\int_\mathbb{R}\varphi_2(t) \varphi_2(t+s\Delta)\, dt +\int_\mathbb{R}\varphi_1(t) \varphi_2(t+s\Delta)\, dt\int_\mathbb{R}\varphi_2(t) \varphi_1(t+s\Delta)\, dt\Bigr).
\end{align*}
\end{theorem}
As was the case in Theorem~\ref{theoremIntro}, statement~(\ref{introAs12}) is more general than statement~(\ref{introAs22}) of Theorem~\ref{theoremIntro2}, but the latter may be convenient as it gives conditions on the decay rate of $\varphi_1$ and $\varphi_2$ at infinity. In relation to Theorem~\ref{theoremIntro2}, it should be mentioned that limit theorems for the sample autocovariances of moving average processes (\ref{CMA}) have been studied in \cite{brandes2018sample,cohen2013central,spangenberg2015limit}.

The paper is organized as follows: Section~\ref{prel} recalls the most relevant concepts in relation to L\'{e}vy processes and the corresponding integration theory. Section~\ref{results} presents the results used to establish Theorem~\ref{theoremIntro} and \ref{theoremIntro2}. In particular, we show that central limit theorems for $Q_n$ and $S_n$ hold under weaker conditions than those given above, and then deduce Theorem~\ref{theoremIntro} and \ref{theoremIntro2} as special cases. Moreover, Section~\ref{results} provides examples demonstrating that the imposed conditions on $\varphi$ (or $\varphi_1$ and $\varphi_2$) are satisfied for CARMA processes, solutions to stochastic delay equations and certain fractional (L\'{e}vy) noise processes. Finally, Section~\ref{proofs} contains proofs of all the statements of the paper together with a few supporting results.

\section{Preliminaries}\label{prel}
In this section we introduce some notation that will be used repeatedly and we recall a few concepts related to L\'{e}vy processes and integration of deterministic functions with respect to them. For a detailed exposition of L\'{e}vy processes and the corresponding integration theory, see \cite{rosSpec,Sato}.

For a given measurable function $f:\mathbb{R}\to \mathbb{R}$ and $p\geq 1$ we write $f\in L^p$ if $\vert f\vert^p$ is integrable with respect to the Lebesgue measure and $f\in L^\infty$ if $f$ is bounded almost everywhere. For a given function $a:\mathbb{Z}\to \mathbb{R}$ (or sequence $(a(t))_{t\in \mathbb{Z}}$) we write $a\in \ell^p$ if $\lVert a\rVert_{\ell^p} := \big(\sum_{t\in \mathbb{Z}}\vert a(t)\vert^p\big)^{1/p}<\infty$ and $a \in \ell^\infty$ if $\lVert a\rVert_{\ell^\infty}:= \sup_{t\in \mathbb{Z}}\vert a(t)\vert < \infty$.

A stochastic process $(L_t)_{t\geq 0}$, $L_0 = 0$, is called a one-sided L\'{e}vy process if it is càdlàg and has stationary and independent increments. The distribution of $(L_t)_{t\geq 0}$ is characterized by $L_1$ as a consequence of the relation $\log \mathbb{E} \exp\{iyL_t\} = t \log \mathbb{E}\exp\{iyL_1\}$. By the L\'{e}vy-Khintchine representation it holds that
\begin{align*}
\log \mathbb{E}e^{iyL_1} = iy\gamma   - \frac{1}{2}\rho^2 y^2 + \int_\mathbb{R}\big(e^{iyx}-1-iyx\mathds{1}_{\vert x \vert \leq 1} \big)\, \nu (dx),\quad y \in \mathbb{R},
\end{align*}
for some $\gamma \in \mathbb{R}$, $\rho^2\geq 0$ and L\'{e}vy measure $\nu$, and hence (the distribution of) $(L_t)_{t\geq 0}$ may be summarized as a triplet $(\gamma,\rho^2,\nu)$. The same holds for a (two-sided) L\'{e}vy process $(L_t)_{t\in \mathbb{R}}$ which is constructed as $L_t = L^1_t\mathds{1}_{t\geq 0} - L^2_{(-t)-}\mathds{1}_{t< 0}$, where $(L^1_t)_{t\geq 0}$ and $(L^2_t)_{t\geq 0}$ are one-sided L\'{e}vy processes which are independent copies.

Let $(L_t)_{t\in \mathbb{R}}$ be a L\'{e}vy process with $\mathbb{E}\vert L_1\vert <\infty$ and $\mathbb{E}L_1 = 0$. Then, for a given measurable function $f:\mathbb{R}\to \mathbb{R}$, the integral $\int_\mathbb{R}f(t)\, dL_t$ is well-defined (as a limit in probability of integrals of simple functions) and belongs to $L^p(\mathbb{P})$, $p\geq 1$, if
\begin{align}\label{intCondition}
\int_\mathbb{R}\int_\mathbb{R} \vert f(t)x\vert^p \wedge (f(t)x)^2\, \nu (dx)\, dt <\infty.
\end{align}
In particular, (\ref{intCondition}) is satisfied if $f \in L^2 \cap L^p$ and $\int_{\vert x \vert >1} \vert x \vert^p\, \nu (dx)<\infty$, the latter condition being equivalent to $\mathbb{E}\vert L_1\vert^p<\infty$. Finally, when (\ref{intCondition}) holds for $p = 2$ we will often make use of the isometry property of the integral map:
\begin{align*}
\mathbb{E}\Bigr[\Bigr(\int_\mathbb{R}f(t)\, dL_t\Bigr)^2\Bigr] = \int_\mathbb{R}f(t)^2\, dt.
\end{align*}

\section{Further results and examples}\label{results}
As in the introduction, it will be assumed throughout that $(L_t)_{t\in \mathbb{R}}$ is a two-sided L\'{e}vy process with $\mathbb{E}L_1 = 0$ and $\mathbb{E}L_1^4<\infty$. Set $\sigma^2 = \mathbb{E}L_1^2$ and $\kappa_4 = \mathbb{E} L_1^4-3\sigma^4$. Moreover, for functions $\varphi,\varphi_1,\varphi_2:\mathbb{R}\to \mathbb{R}$ in $L^2$ define
\begin{align}\label{MAprocesses}
X_t = \int_\mathbb{R} \varphi (t-s)\, dL_s, \quad \text{and}\quad X^i_t = \int_\mathbb{R}\varphi_i(t-s)\, dL_s
\end{align}
for $t \in \mathbb{R}$ and $i=1,2$. We will be interested in the quantities
\begin{align}\label{relevantQuantity}
S_n = \sum_{t=1}^n X^1_{t\Delta}X^2_{t\Delta}\quad \text{and}\quad Q_n = \sum_{t,s = 1}^n b(t-s)X_{t\Delta}X_{s\Delta}
\end{align}
for a given $\Delta>0$ and an even function $b:\mathbb{Z}\to \mathbb{R}$. Our main results, Theorem~\ref{Qform} and \ref{quadraticForms}, provide a central limit theorem for the quantities in (\ref{relevantQuantity}) and are more general than Theorem~\ref{theoremIntro} and \ref{theoremIntro2} which were presented in Section~\ref{intro}. Before the formulations we define the autocovariance function of $(X_t)_{t\in \mathbb{R}}$,
\begin{align}\label{autCovQform}
\gamma_X (h) = \mathbb{E}[X_0X_h] = \sigma^2 \int_\mathbb{R}\varphi (t)\varphi (t+h)\, dt,\quad h \in \mathbb{R},
\end{align}
as well as the autocovariance (crosscovariance) functions of $(X^1_t)_{t\in \mathbb{R}}$ and $(X^2_t)_{t\in \mathbb{R}}$,
\begin{align}\label{autCrossCov}
\gamma_{ij}(h) = \mathbb{E}\big[X^i_0X^j_h\big] = \sigma^2\int_\mathbb{R}\varphi_i(t)\varphi_j(t+h)\, dt,\quad h \in \mathbb{R}.
\end{align}
\begin{theorem}\label{Qform} Suppose that the following conditions hold:
\begin{enumerate}[(i)]
\item\label{firstAs1} $\int_\mathbb{R}\vert \varphi_i(t)\varphi_i(t+\cdot\Delta)\vert\, dt \in \ell^{\alpha_i}$ for $i=1,2$ and $\alpha_1,\alpha_2\in [1,\infty]$ with $1/\alpha_1 + 1/\alpha_2 = 1$.
\item\label{firstAs3} $\int_\mathbb{R}\vert \varphi_1(t)\varphi_2(t+\cdot\Delta)\vert\, dt \in \ell^2$. 
\item\label{firstAs2} $\big( t\mapsto \lVert \varphi_1(t+\cdot\Delta)\varphi_2(t+\cdot\Delta)\rVert_{\ell^1}\big)\in L^2([0,\Delta])$.
\end{enumerate}

Then, as $n\to \infty$, $(S_n-\mathbb{E}S_n)/\sqrt{n}$ tends to a Gaussian random variable with mean zero and variance
\begin{align}\label{SnVariance}
\begin{aligned}
\eta^2 =&\, \kappa_4 \int_0^\Delta \Bigr(\sum_{s\in \mathbb{Z}} \varphi_1 (t+s\Delta) \varphi_2 (t+s\Delta)\Bigr)^2\, dt +\sum_{s\in \mathbb{Z}} \gamma_{11}(s\Delta)\gamma_{22}(s\Delta)\\
&+\sum_{s\in \mathbb{Z}} \gamma_{12}(s\Delta)\gamma_{21}(s\Delta).
\end{aligned}
\end{align}
\end{theorem}
\begin{remark}\label{GaussianRemark}
An inspection of the proof of Theorem~\ref{Qform} will reveal that assumption~(\ref{firstAs2}) is not needed in case $(L_t)_{t\in \mathbb{R}}$ is a Brownian motion. In this situation, $\kappa_4 = 0$ and the variance formula (\ref{SnVariance}) reduces to
\begin{align*}
\eta^2 =\sum_{s\in \mathbb{Z}} \gamma_{11}(s\Delta)\gamma_{22}(s\Delta)+\sum_{s\in \mathbb{Z}} \gamma_{12}(s\Delta)\gamma_{21}(s\Delta).
\end{align*}
Note that, since the following results in this section rely on Theorem~\ref{Qform}, the corresponding assumptions may be relaxed accordingly when we are in the Gaussian setting.
\end{remark}

Loosely speaking, assumptions~(\ref{firstAs1})-(\ref{firstAs3}) of Theorem~\ref{Qform} concern summability of continuous-time convolutions. Hence, by relying on a modification of Young's convolution inequality, Theorem~\ref{theoremIntro2} can be shown to be a special case Theorem~\ref{Qform} (see Lemma~\ref{youngConvolution} and the following proof of Theorem~\ref{theoremIntro2} in Section~\ref{proofs}). Example~\ref{sampleAutocovariances} and \ref{leastSquares} are possible applications of Theorem~\ref{Qform}.
\begin{example}\label{sampleAutocovariances}
Let $n,m \in \mathbb{N}$ with $m<n-1$, define the sample autocovariance of $(X_t)_{t\in \mathbb{R}}$ based on $X_\Delta,X_{2\Delta},\dots, X_{n\Delta}$ up to lag $m$ as
\begin{align}\label{sampleACF}
\hat{\gamma}_n (j) = n^{-1}\sum_{t=1}^{n-j}X_{t\Delta}X_{(t+j)\Delta},\quad j=1,\dots, m,
\end{align}
and set $\hat{\gamma}_n = (\hat{\gamma}_n(1),\dots, \hat{\gamma}_n(m))^T$. Moreover, let $\tilde{\varphi}(t) = (\varphi (t+\Delta),\dots, \varphi (t+m\Delta))^T$ and $\gamma_s = (\gamma_X((s+1)\Delta),\dots,\gamma_X((s+m)\Delta))^T$ using the notation as in (\ref{MAprocesses}) and (\ref{autCovQform}). Then for a given $\alpha = (\alpha_1,\dots,\alpha_m)^T \in \mathbb{R}^m$, it holds that
\begin{align}\label{ACFrelation}
\alpha^T \hat{\gamma}_n-\alpha^T\gamma_0 = n^{-1} \sum_{t=1}^{n} \Bigr(X^1_{t\Delta}X^2_{t\Delta}-\mathbb{E}[X^1_0X^2_0]\Bigr) + O_p\big(n^{-1}\big),
\end{align}
where $(X^1_t)_{t\in \mathbb{R}}$ and $(X^2_t)_{t\in \mathbb{R}}$ are given by (\ref{MAprocesses}) with $\varphi_1 = \varphi$ and $\varphi_2(t) = \alpha^T \tilde{\varphi}(t)$. Here $O_p(n^{-1})$ in (\ref{ACFrelation}) means that the equality holds up to a term $\varepsilon_n$ which is stochastically bounded by $n^{-1}$ (that is, $(n \varepsilon_n)_{n\in \mathbb{N}}$ is tight). Then if
\begin{align}\label{sampleACFassumption}
\int_\mathbb{R}\vert \varphi (t)\varphi (t+\cdot \Delta)\vert\, dt \in \ell^2 \quad \text{and}\quad
\big(t\mapsto \lVert \varphi (t+\cdot \Delta) \rVert_{\ell^2}^2 \big)\in L^2([0,\Delta]),
\end{align}
assumptions~(\ref{firstAs1})-(\ref{firstAs2}) of Theorem~\ref{Qform} hold and we deduce that $\alpha^T \sqrt{n}(\hat{\gamma}_n-\gamma_0)$ converges in distribution to a centered Gaussian random variable with variance $\alpha^T\Sigma \alpha$, where 
\begin{align*}
\Sigma = \kappa_4\int_0^\Delta K(t)K(t)^T\, dt + \sum_{s\in \mathbb{Z}} (\gamma_s+\gamma_{-s}) \gamma_s^T,\quad K(t) := \sum_{s\in \mathbb{Z}} \varphi (t+s\Delta)\tilde{\varphi}(t+s\Delta).
\end{align*}
By the Cramér-Wold theorem we conclude that $\sqrt{n}(\hat{\gamma}_n-\gamma_0)$ converges in distribution to a centered Gaussian vector with covariance matrix $\Sigma$. This type of central limit theorem for the sample autocovariances of continuous-time moving averages was established in \cite{cohen2013central} under the same assumptions on $\varphi$ as imposed above.
\end{example}

\begin{example}\label{leastSquares}
Motivated by the discussion in the introduction, this example will illustrate how Theorem~\ref{Qform} can be applied to show asymptotic normality of the (adequately normalized) derivative of a least squares objective function. Fix $k \in \mathbb{N}$, let $v:\mathbb{R}\to \mathbb{R}^k$ be a differentiable function with derivative $v'$ and consider
\begin{align}\label{objective}
\ell_n(\theta)  = \sum_{t=k+1}^{n} \big(X_{j\Delta}-v(\theta)^TX(j)\big)^2,\quad \theta \in \mathbb{R},
\end{align}
where $X(t) = (X_{(t-1)\Delta},\dots, X_{(t-k)\Delta})^T$. In this case
\begin{align*}
\ell_n' (\theta) = -2\sum_{t=k+1}^n \big(X_{t\Delta}-v(\theta)^T X(t)\big)v'(\theta)^T X(t),
\end{align*}
and hence it is of the same form as $S_n$ in (\ref{relevantQuantity}) with $\varphi_1 (t) = \big(-1,v(\theta)^T\big)\tilde{\varphi}(t)$ and $\varphi_2(t) = \big(0,2v'(\theta)^T \big)\tilde{\varphi}(t)$, where $\tilde{\varphi} (t) = (\varphi (t),\varphi (t-\Delta),\dots, \varphi (t-k\Delta))^T$. Suppose that $v (\theta_0)$ coincides with the vector of coefficients of the $L^2( \mathbb{P})$-projection of $X_{(k+1)\Delta}$ onto the linear span of $X_{k\Delta},\dots, X_{\Delta}$ for some $\theta_0\in \mathbb{R}$. In this case $\mathbb{E}\ell_n'(\theta_0) = 0$, and if (\ref{sampleACFassumption}) holds it thus follows from Theorem~\ref{Qform} that $\ell_n'(\theta_0)/\sqrt{n}$ converges in distribution to a centered Gaussian random variable.
\end{example}

%Proposition~\ref{HyperbolicDecay} provides an easy-to-check condition for applying Theorem~\ref{Qform}.

%\begin{proposition}\label{HyperbolicDecay}
%Suppose that
%\begin{align}\label{functionDecay2}
%\sup_{t\in \mathbb{R}} \vert t\vert^{\alpha_i}\vert \varphi_i(t)\vert < \infty
%\end{align}
%for some $\alpha_1,\alpha_2 \in (0,1)$ with $\alpha_1 + \alpha_2 >3/2$. Then assumptions~(\ref{firstAs1})-(\ref{firstAs2}) of Theorem~\ref{Qform} are satisfied.
%\end{proposition}
Theorem~\ref{quadraticForms} is our most general result concerning the limiting behavior of $(Q_n -\mathbb{E}Q_n)/\sqrt{n}$ as $n\to \infty$. For notational convenience we will, for given $a:\mathbb{Z}\to\mathbb{R}$ and $f:\mathbb{R}\to \mathbb{R}$, set
\begin{align}\label{starConv}
(a\star f) (t) := \sum_{s\in \mathbb{Z}} a(s)f (t-s\Delta)
\end{align}
for any $t\in \mathbb{R}$, such that $\sum_{s\in \mathbb{Z}} \vert a(s)f (t-s\Delta)\vert < \infty$. If $a$ and $f$ are non-negative, the definition in (\ref{starConv}) is used for all $t\in \mathbb{R}$. Moreover, we write $\vert a \vert (t) = \vert a(t) \vert$ and $\vert f \vert (t) = \vert f(t)\vert$.

\begin{theorem}\label{quadraticForms} Suppose that the following statements hold:
\begin{enumerate}[(i)]
\item\label{as1} There exist $\alpha,\beta\in [1,\infty]$ with $1/\alpha + 1/\beta = 1$, such that $\int_\mathbb{R}\vert \varphi (t)\varphi (t+\cdot \Delta)\vert\, dt \in \ell^\alpha$ and $\int_\mathbb{R} (\vert b \vert \star \vert \varphi\vert) (t)\, (\vert b \vert \star \vert \varphi\vert) (t+\cdot \Delta)\, dt\in \ell^\beta$.

%\item $\sum_{s\in \mathbb{Z}}\big( \sum_{u \in \mathbb{Z}}\vert b(u)\vert \int_\mathbb{R}\vert \varphi (t) \varphi ((s+u)\Delta + t)\vert\, dt\big)^2<\infty$.\label{as2}
\item $\int_\mathbb{R}\vert \varphi (t)\vert\, (\vert b \vert \star \vert \varphi \vert)(t+\cdot \Delta)\, dt \in \ell^2$.\label{as2}

\item $\big(t\mapsto \lVert \varphi (t+\cdot \Delta)\, (\vert b \vert \star \vert \varphi \vert)(t+\cdot \Delta) \rVert_{\ell^1} \big) \in L^2([0,\Delta])$.\label{as3}

%\item The function $t \mapsto \sum_{s,u\in \mathbb{Z}} \vert b(u)\varphi (s\Delta + t)\varphi ((s+u)\Delta +t)\vert$ belongs to $L^2([0,\Delta])$.\label{as3}
%\item The function $t\mapsto \sum_{s\in \mathbb{Z}}\vert b(s) \varphi (s\Delta+t)\vert$ belongs to $L^2 \cap L^4$.\label{as4}
\end{enumerate}
Then, as $n\to \infty$, $(Q_n - \mathbb{E}Q_n)/\sqrt{n}$ converges in distribution to a Gaussian random variable with mean zero and variance
%\begin{align}\label{quadraticFormVariance}
%\begin{aligned}
%\eta^2 =&\ \kappa_4\sum_{t,s,u\in \mathbb{Z}}b(s)b(u)\int_\mathbb{R}\varphi (v)\varphi (s\Delta + v) \varphi (t\Delta +v) \varphi ((t+u)\Delta + v)\, dv\\
%& + \sum_{t\in \mathbb{Z}}\gamma_X (t\Delta) \sum_{s,u \in \mathbb{Z}}b(s)b(u) \gamma_X ((t+s+u)\Delta) + \sum_{t\in \mathbb{Z}} \Bigr(\sum_{s\in \mathbb{Z}}b(s)\gamma_X ((t+s)\Delta) \Bigr)^2.
%\end{aligned}
%\end{align}
\begin{align}\label{quadraticFormVariance}
\begin{aligned}
\eta^2 =&\ \kappa_4\int_0^\Delta \Bigr(\sum_{s\in \mathbb{Z}} \varphi (t+s\Delta)\, (b\star \varphi)(t+s\Delta)\Bigr)^2\, dt + 2\lVert (b\star \gamma_X)(\cdot \Delta) \rVert_{\ell^2}^2.
\end{aligned}
\end{align}
\end{theorem}

\begin{remark}\label{QnApprox}
The idea in the proof of Theorem~\ref{quadraticForms} is to approximate $Q_n$ by $S_n$ with $\varphi_1 = \varphi$ and $\varphi_2 = b \star \varphi$. The conditions imposed in Theorem~\ref{quadraticForms} correspond to assuming that $\varphi$ and $\vert b \vert \star \vert \varphi\vert$ satisfy (\ref{firstAs1})-(\ref{firstAs2}) of Theorem~\ref{Qform}. In particular, these conditions ensure that $S_n$ is well-defined and that Theorem~\ref{Qform} applies to this choice of $\varphi_1$ and $\varphi_2$. The only lacking part in order to deduce Theorem~\ref{quadraticForms} from Theorem~\ref{Qform} is to show that $S_n$ is in fact a proper approximation of $Q_n$ in the sense that $\text{Var}(Q_n-S_n)/n\to 0$ as $n\to \infty$, but this is verified in Section~\ref{proofs} where the proofs of the stated results can be found.
\end{remark}

\begin{remark}\label{SufficientQuadForm}
Note that for any $s\in \mathbb{Z}$ with $b(s)\neq 0$, it holds that 
\begin{align}\label{convPhiIneq}
\vert \varphi (t)\vert\leq \vert b(s)\vert^{-1}(\vert b\vert \star \vert \varphi \vert) (t+s\Delta)
\end{align}
for all $t\in \mathbb{R}$. This fact ensures that assumptions~(\ref{as1})-(\ref{as2}) of Theorem~\ref{quadraticForms} hold if there exists $\beta\in [1,2]$ such that
\begin{align}\label{conditionQform}
\int_\mathbb{R}(\vert b\vert \star \vert \varphi \vert) (t)\, (\vert b\vert \star \vert \varphi \vert)(t+\cdot \Delta)\, dt \in \ell^\beta.
\end{align}
(Here we exclude the trivial case $b\equiv 0$.) Indeed, if (\ref{conditionQform}) is satisfied we can choose $\alpha \geq \beta$ such that $1/\alpha + 1/\beta = 1$ and then assumptions (\ref{as1})-(\ref{as2}) are met due to the inequality (\ref{convPhiIneq}) and the fact that $\ell^\beta \subseteq \ell^\alpha \cap \ell^2$.
\end{remark}

\begin{remark}
We will now briefly comment on the conditions of Theorem~\ref{theoremIntro} and Theorem~\ref{theoremIntro2}, particularly on sufficient conditions for applying Theorem~\ref{Qform} and Theorem~\ref{quadraticForms}. We will restrict our attention to assumptions of the type
\begin{align}\label{remarkAssump}
\big(t\mapsto\lVert \psi (t+\cdot \Delta)\rVert_{\ell^\kappa}^\kappa\big) \in L^2([0,\Delta]),
\end{align}
where $\psi:\mathbb{R}\to \mathbb{R}$ is a measurable function and $\kappa \geq 1$. First of all, note that the weaker condition $\big(t\mapsto\lVert \psi (t+\cdot \Delta)\rVert_{\ell^\kappa}^\kappa\big) \in L^1([0,\Delta])$ is satisfied if and only if $\psi \in L^\kappa$, and condition (\ref{remarkAssump}) implies $\psi \in L^{2\kappa}$. In particular, a necessary condition for (\ref{remarkAssump}) to hold is that $\psi \in L^\kappa \cap L^{2\kappa}$. On the other hand, one may decompose $\lVert \psi (t+\cdot \Delta) \rVert_{\ell^\kappa}^\kappa$ as
\begin{align}\label{remarkDecomposition}
\lVert\psi (t+\cdot \Delta) \rVert_{\ell^\kappa}^\kappa &= \sum_{s=-M}^M \vert \psi (t+s\Delta)\vert^\kappa+ \sum_{s=M+1}^\infty \big(\vert \psi (t+s\Delta)\vert^\kappa + \vert \psi (t-s\Delta)\vert^\kappa\big)
\end{align}
for any $M \in \mathbb{N}$. The first term on right-hand side of (\ref{remarkDecomposition}) belongs to $L^2([0,\Delta])$ (viewed as a function of $t$) if $\psi \in L^{2\kappa}$. If in addition $\psi \in L^\kappa$, the second term on the right-hand tends to zero as $M\to \infty$ for (Lebesgue almost) all $t\in [0,\Delta]$. If this could be assumed to hold uniformly across all $t$, that is, if the second term belongs to $L^\infty ([0,\Delta])$ for a sufficiently large $M$, then (\ref{remarkAssump}) would be satisfied. Therefore, loosely speaking, the difference between $L^\kappa\cap L^{2\kappa}$ and the space of functions satisfying (\ref{remarkAssump}) consists of functions $\psi$ where the second term in (\ref{remarkDecomposition}) tends to zero pointwise, but not uniformly, in $t$ as $M\to \infty$. Ultimately, this is a condition on the behavior of the tail of the function between grid points. For instance, if there exists a sequence $(\psi_s)_{s\in \mathbb{Z}}$ in $\ell^\kappa$ such that $\sup_{t\in [0,\Delta]}\vert \psi (t\pm s\Delta) \vert\leq \psi_s$ for all sufficiently large $s$, then (\ref{remarkAssump}) holds. An assumption such as (\ref{remarkAssump}) seems to be necessary and is the cost of considering a continuous-time process only on a discrete-time grid. In \cite{cohen2013central}, where they prove a central limit theorem for the sample autocovariance of a continuous-time moving average in a low-frequency setting, a similar condition is imposed.
%In particular, if $\psi = \sum_{s\in \mathbb{Z}}\psi_s \mathds{1}_{(s\delta,(s+1)\delta]}$ for some $\delta>0$, (\ref{remarkAssump}) holds if and only if $(\psi_s)_{s\in \mathbb{Z}}$ belongs to $\ell^\kappa$. For this reason, Corollary~\ref{SufficientIntegrability} is applicable if $(X^i_t)_{t\in \mathbb{R}}$ is of a discrete-time moving average type
%\begin{align*}
%X^i_t = \sum_{s\in \mathbb{Z}} \psi^i_s \varepsilon (\delta)_{t-s},\quad \varepsilon (\delta)_t:= L_t-L_{t-\delta},
%\end{align*}
%for a given sequence $(\psi_s^i)_{s\in \mathbb{Z}}$ in $\ell^{\alpha_i}$. A rather general easy-to-check condition for applying Corollary~\ref{SufficientIntegrability} is given in Example~\ref{decayIntegrability} below.
\end{remark}

In the following examples we will put some attention on concrete specifications of moving average processes, where the behavior of the corresponding kernel is known, and hence Theorem~\ref{theoremIntro} and \ref{theoremIntro2} may be applicable. 

\begin{example}\label{CARMAprocesses}
Fix $p \in \mathbb{N}$ and let $P(z) = z^p + a_1z^{p-1}+ \cdots + a_p$ and $Q(z) = b_0 + b_1z + \cdots + b^{p-1}z^{p-1}$, $z\in \mathbb{C}$, be two real polynomials where all the zeroes of $P$ are contained in $\{z \in \mathbb{C}\, :\, \text{Re}(z)<0\}$. Moreover, let $q \in \mathbb{N}_0$ with $q<p$ and suppose that $b_q = 1$ and $b_k = 0$ for $q<k\leq p-1$. Finally, define
\begin{align*}
A = \begin{bmatrix} 0 & 1 & 0 & \cdots & 0 \\
0 & 0 & 1 & \cdots & 0 \\
\vdots & \vdots & \vdots & \ddots & \vdots \\
0 & 0 & 0 &\cdots & 1 \\
- a_p & - a_{p-1} & -a_{p-2} & \cdots & -a_1
\end{bmatrix},\quad b = \begin{bmatrix} b_0 \\ b_1 \\ \vdots \\ b_{p-2} \\ b_{p-1} \end{bmatrix} \quad \text{and} \quad e_p = \begin{bmatrix} 0 \\ 0 \\ \vdots \\ 0 \\ 1 \end{bmatrix}.
\end{align*}
Then the corresponding (causal) CARMA($p,q$) process $(X_t)_{t\in \mathbb{R}}$ is given by
\begin{align}\label{CARMAdef}
X_t = \int_{-\infty}^t b^T e^{A(t-u)} e_p\, dL_u,\quad t \in \mathbb{R}.
\end{align}
(See \cite[Remark~3.2]{marquardt2007multivariate}.) The definition in (\ref{CARMAdef}) is based on a state-space representation of the more intuitive formal differential equation
\begin{align}\label{formalDiffEq}
P(D)X_t = Q(D)DL_t,\quad t \in  \mathbb{R},
\end{align}
where $D$ denotes differentiation with respect to time. Equation (\ref{formalDiffEq}) should be compared to the corresponding representation of an ARMA process in terms of the backward-shift operator. Since it can be shown that the eigenvalues of $A$ correspond to the roots of $P$, the kernel $\varphi:t \mapsto \mathds{1}_{[0,\infty)}(t)b^Te^{At}e_p$ is exponentially decaying at infinity. Combining this with the (absolute) continuity of $\varphi$ on $[0,\infty)$ ensures that the kernel belongs to $L^\infty$ as well. In particular, this shows that Theorem~\ref{theoremIntro}(\ref{introAs1}) holds as long as $b\in \ell^2$. For more on CARMA processes, we refer to \cite{brockwell2001levy,brockwell2011estimation,brockwell2009existence}.
\end{example}

\begin{example}\label{SDDEexample}
Let $\eta$ be a finite signed measure on $[0,\infty)$ and suppose that
\begin{align*}
z+\int_{[0,\infty)} e^{zu}\, \eta (du) \neq 0
\end{align*}
for every $z \in \mathbb{C}$ with $\text{Re}(z)\leq 0$. Then it follows from \cite[Theorem~3.4]{contARMA} that the unique stationary solution $(X_t)_{t\in \mathbb{R}}$ to the corresponding stochastic delay differential equation
\begin{align*}
dX_t = \int_{[0,\infty)}X_{t-s}\, \eta (ds)\, dt + dL_t,\quad t\in \mathbb{R},
\end{align*}
takes the form $X_t = \int_{-\infty}^t \varphi(t-s)\, dL_s$, where $\varphi:\mathbb{R}\to \mathbb{R}$ is characterized as the unique $L^2$ function satisfying $\varphi (t) = 0$ for $t<0$ and
\begin{align*}
\varphi (t) = 1 + \int_0^t \int_{[0,\infty)} \varphi (s-u)\, \eta (du)\, ds
\end{align*}
for $t\geq 0$. Consequently, it follows form the integration by parts formula that
\begin{align}\label{supSDDEbound}
\begin{aligned}
\sup_{t\geq 0} t^p\vert \varphi (t)\vert \leq &\, p\int_0^\infty  t^{p-1}\vert \varphi (t)\vert\, dt + 2^p \vert \eta \vert ([0,\infty))\int_0^\infty  t^p\vert \varphi (t)\vert\, dt\\
& + 2^p\int_{[0,\infty)}t^p\, \vert \eta \vert (dt)\int_0^\infty \vert \varphi (t)\vert\, dt
\end{aligned}
\end{align}
for a given $p\geq 1$. Here $\vert \eta \vert$ is the variation measure of $\eta$. If one assumes that $\vert \eta \vert$ has moments up to order $p+1$, that is,
\begin{align*}
\int_{[0,\infty)} t^{p+1} \vert \eta \vert (dt) < \infty,
\end{align*}
it follows by \cite[Lemma~3.2]{contARMA} that the measure $\vert \varphi (t)\vert \, dt$ is finite and has moments up to order $p$. Consequently, under this assumption we have that $\sup_{t\geq 0} t^p\, \vert \varphi (t)\vert<\infty$ by (\ref{supSDDEbound}) and Theorem~\ref{theoremIntro}(\ref{introAs2}) holds as long as $\sup_{t\in \mathbb{Z}} \vert t\vert^{1/2+\delta}\vert b(t)\vert < \infty$ for some $\delta >0$.
\end{example}

\begin{example}\label{long-memory}
Suppose that $(X_t)_{t\in \mathbb{R}}$ is given by (\ref{MAprocesses}) with
\begin{align*}
\varphi (t) = \frac{1}{\Gamma (1+d)}\big[t_+^{d}- (t-1)_+^{d}\big],\quad t \in \mathbb{R},
\end{align*}
and $d \in (0,1/4)$. (Here $\Gamma (1+d) = \int_0^\infty u^{d}e^{-u}\, du$ is the Gamma function at $1+d$.) In other words, we assume that $(X_t)_{t\in \mathbb{R}}$ is a fractional L\'{e}vy noise with parameter $d$. Recall that $\gamma_X(h) \sim c h^{2d-1}$ as $h\to \infty$ for a suitable constant $c>0$ (see, e.g., \cite[Theorem~6.3]{Tina}), and hence we are in a setup where
\begin{align*}
\sum_{s \in \mathbb{Z}} \vert \gamma_X(s\Delta)\vert = \infty, \quad \text{but}\quad \sum_{s \in \mathbb{Z}} \gamma_X(s \Delta)^2 < \infty.
\end{align*}
Moreover, it is shown in \cite[Theorem~A.1]{cohen2013central} that $(X_{t\Delta})_{t\in \mathbb{Z}}$ is not strongly mixing. However, Theorem~\ref{theoremIntro} and \ref{theoremIntro2} may still be applied in this setup, since $\varphi$ is vanishing on $(-\infty,0)$, continuous on $\mathbb{R}$, and $\varphi (t) \sim d\, t^{d-1}/\Gamma (1+d)$ as $t\to \infty$.
\end{example}

\section{Proofs}\label{proofs}
The first observation will be used in the proof of Theorem~\ref{Qform}.

\begin{lemma}\label{fourthMoment}
Let $g_1,g_2,g_3,g_4:\mathbb{R}\to \mathbb{R}$ be functions in $L^2 \cap L^4$. Then it holds that
\begin{align}\label{keyRelation}
\begin{aligned}
\MoveEqLeft\mathbb{E}\prod_{j=1}^4 \int_\mathbb{R}g_j(u)\, dL_u \\
&= \kappa_4 \int_\mathbb{R} \prod_{j=1}^4 g_j(u)\, du + \sigma^4 \biggr(\int_\mathbb{R}g_1(u) g_2(u)\, du\, \int_\mathbb{R}g_{3}(u)g_4(u)\, du\\
&+\int_\mathbb{R}g_1(u) g_3(u)\, du\, \int_\mathbb{R}g_2(u)g_4(u)\, du + \int_\mathbb{R}g_1(u) g_4(u)\, du\, \int_\mathbb{R}g_2(u)g_3(u)\, du\biggr).
\end{aligned}
\end{align}
\end{lemma}
\begin{proof}
Set $Y_i = \int_\mathbb{R}g_i(u)\, dL_u$. Then, using \cite[Proposition~4.2.2]{giraitis2012large}, we obtain that
\begin{align}\label{shirFormulas}
\begin{aligned}
\mathbb{E}[Y_1Y_2Y_3Y_4] =& \text{Cum}(Y_1,Y_2,Y_3,Y_4) + \mathbb{E}[Y_1Y_2]\mathbb{E}[Y_3Y_4]\\
 &+ \mathbb{E}[Y_1Y_3]\mathbb{E}[Y_2Y_4] + \mathbb{E}[Y_1Y_4]\mathbb{E}[Y_2Y_3],
 \end{aligned}
\end{align}
where
\begin{align*}
\text{Cum} (Y_1,Y_2,Y_3,Y_4) = \frac{\partial^4}{\partial u_1 \cdots \partial u_4}\log \mathbb{E}e^{i(u_1Y_1 +\cdots +u_4Y_4)}\biggr\vert_{u_1= \cdots = u_4 = 0}.
\end{align*}
Set $\psi_L (u) = \log \mathbb{E}e^{iuL_1}$. It follows from the L\'{e}vy-Khintchine representation that we can find a constant $C>0$ such that $\vert\psi_L^{(1)}(u)\vert \leq C\vert u \vert$ and $\vert\psi_L^{(m)}(u)\vert\leq C$ for $m=2,3,4$. (Here $\psi^{(m)}_L$ is the $m$-th derivative of $\psi_L$.) Using this together with the representation
\begin{align*}
\log\mathbb{E}e^{i(u_1Y_1 + \cdots + u_4Y_4)} = \int_\mathbb{R}\psi_L(u_1g_1(t) + \cdots + u_4g_4(t))\, dt,
\end{align*}
see \cite{rosSpec}, we can interchange differentiation and integration to obtain
\begin{align*}
\MoveEqLeft\text{Cum}(Y_1,Y_2,Y_3,Y_4) \\&= \int_\mathbb{R}\psi^{(4)}_L(u_1g_1(t) + \cdots + u_4g_4(t))\prod_{j=1}^4g_j(t)\, dt\biggr\vert_{u_1=\cdots = u_4 = 0}= \kappa_4 \int_\mathbb{R}\prod_{j=1}^4g_j(t)\, dt.
\end{align*}
By combining this observation with the fact that $\mathbb{E}[Y_jY_k] = \sigma^2 \int_\mathbb{R}g_j(u)g_k(u)\, du$ (using the isometry property), the result is an immediate consequence of (\ref{shirFormulas}).
\end{proof}
\begin{remark}
In case $g_0 = g_1=g_2=g_3$, Lemma~\ref{fourthMoment} collapses to \cite[Lemma~3.2]{cohen2013central}, and if $(L_t)_{t\in \mathbb{R}}$ is a Brownian motion we have that $\kappa_4 = 0$ and the result is a special case of Isserlis' theorem.
\end{remark}

We are now ready to prove Theorem~\ref{Qform}.

\begin{proof}[Proof of Theorem~\ref{Qform}]
The proof goes by approximating $(X^1_{t\Delta}X^2_{t\Delta})_{t\in \mathbb{Z}}$ by a $k$-dependent sequence (cf. \cite[Definition~6.4.3]{brockDavis}), to which we can apply a classical central limit theorem. Fix $m>0$, and set $\varphi_i^m = [(-m)\vee \varphi_i\wedge m]\mathds{1}_{[-m,m]}$ and
\begin{align*}
X_t^{i,m} =\int_\mathbb{R} \varphi_i^m (t-s)\, dL_s= \int_{t-m}^{t+m} \varphi_i^m (t-s)\, dL_s,\quad t \in \mathbb{R},
\end{align*}
for $i=1,2$. Furthermore, set
\begin{align*}
S_n^m = \sum_{t=1}^n X^{1,m}_{t\Delta}X^{2,m}_{t\Delta},\quad n \in \mathbb{N}.
\end{align*}
Note that since $\varphi_i^m\in L^2 \cap L^4$ and $\varphi_i^m (t) = 0$ when $\vert t \vert >m$, $(X^{1,m}_{t\Delta}X^{2,m}_{t\Delta})_{t\in \mathbb{Z}}$ is a $k(m)$-dependent sequence of square integrable random variables, where $k(m)  = \inf\{n \in \mathbb{N}\, :\, n\geq 2m/\Delta \}$. Hence, we can apply \cite[Theorem~6.4.2]{brockDavis} to deduce that
\begin{align*}
\frac{S_n^m-\mathbb{E}S_n^m}{\sqrt{n}} \overset{d}{\to} Y_m, \quad n \to \infty,
\end{align*}
where $Y_m$ is a Gaussian random variable with mean zero and variance 
\begin{align}\label{etaTrunc}
\eta_m^2= \sum_{s=-k(m)}^{k(m)} \gamma_{X^{1,m}X^{2,m}}(s\Delta).
\end{align}
Here $\gamma_{X^{1,m}X^{2,m}}$ denotes the autocovariance function of $(X^{1,m}_tX^{2,m}_t)_{t\in \mathbb{R}}$. Next, we need to argue that $\eta^2_m \to \eta^2$ with $\eta^2$ given by (\ref{SnVariance}). Since $\varphi_i^m \in L^2\cap L^4$ we can use Lemma~\ref{fourthMoment} to compute $\gamma_{X^{1,m}X^{2,m}}(s\Delta)$ for each $s\in \mathbb{Z}$:
\begin{align}\label{autocovTrunc}
\begin{aligned}
\MoveEqLeft\gamma_{X^{1,m}X^{2,m}}(s\Delta) \\
=&\, \kappa_4 \int_\mathbb{R} \varphi_1^m(t) \varphi_2^m(t)\varphi_1^m (t+s\Delta)\varphi_2^m (t+s\Delta)\, dt+ \sigma^4\int_\mathbb{R}\varphi_1^m(t)\varphi_1^m(t+s\Delta)\, dt\\
&\cdot \int_\mathbb{R}\varphi_2^m(t)\varphi_2^m(t+s\Delta)\, dt +  \sigma^4\int_\mathbb{R}\varphi_1^m(t)\varphi_2^m(t+s\Delta)\, dt \cdot \int_\mathbb{R}\varphi_2^m(t)\varphi_1^m(t+s\Delta)\, dt.
\end{aligned}
\end{align}
Note that $\sigma^4\int_\mathbb{R}\varphi_i^m (t) \varphi_j^m(t+s\Delta)\, dt \to \gamma_{ij}(s\Delta)$, since $\varphi_i^m \to \varphi_i$ in $L^2$. By using assumption~(\ref{firstAs2}) and that $F:t\mapsto \sum_{s\in \mathbb{Z}}\vert \varphi_1(t+s\Delta)\varphi_2(t+s\Delta)\vert$ is a periodic function with period $\Delta$ we establish as well that
\begin{align}\label{aa}
\begin{aligned}
\MoveEqLeft\sum_{s \in \mathbb{Z}} \int_\mathbb{R} \vert \varphi_1(t)\varphi_2(t)\varphi_1(t+s\Delta) \varphi_2(t+s\Delta) \vert \, dt \\
&= \sum_{s\in \mathbb{Z}} \int_{s\Delta}^{(s+1)\Delta}\vert \varphi_1(t) \varphi_2(t)\vert F(t)\, dt = \int_0^\Delta F(t)^2\, dt<\infty.
\end{aligned}
\end{align}
In particular, Lebesgue's theorem on dominated convergence implies
\begin{align*}
\int_\mathbb{R} \varphi_1^m(t) \varphi_2^m(t)\varphi_1^m (t+s\Delta)\varphi_2^m (t+s\Delta)\, dt \to \int_\mathbb{R} \varphi_1(t) \varphi_2(t)\varphi_1 (t+s\Delta)\varphi_2 (t+s\Delta)\, dt.
\end{align*}
Combining these observations with (\ref{autocovTrunc}) shows that $\gamma_{X^{1,m}X^{2,m}}(s\Delta)\to \gamma_s$ for each $s\in \mathbb{Z}$, where 
\begin{align*}
\gamma_s=&\, \kappa_4 \int_\mathbb{R} \varphi_1(t) \varphi_2(t)\varphi_1 (t+s\Delta)\varphi_2 (t+s\Delta)\, dt + \gamma_{11}(s\Delta)
\gamma_{22}(s\Delta) +  \gamma_{12}(s\Delta)
\gamma_{21}(s\Delta)
\end{align*}
%\begin{align*}
%T. Now we use Lemma~\ref{fourthMoment} to obtain the following bound for all $m>0$ and $h\in \mathbb{R}$:
%\begin{align}\label{gammaDom}
%\begin{aligned}
%\MoveEqLeft\vert \gamma_{X^{1,m}X^{2,m}}(h) \vert \\
%\leq & \vert \kappa_4 \vert \int_\mathbb{R}\vert f_1(h+u)f_2(h+u)f_1(u) f_2(u) \vert \, du + \sigma^4 \int_\mathbb{R}\vert f_1(u)f_1(h+u) \vert\, du \\
%&\cdot \int_\mathbb{R}\vert f_2(u)f_2(h+u)\vert\, du  + \sigma^4 \int_\mathbb{R}\vert f_1(u) f_2(h+u)\vert\, du \cdot \int_\mathbb{R}\vert f_2(u) f_1(h+u)\vert\, du .
%\end{aligned}
%\end{align}
It follows as well from (\ref{autocovTrunc}) that
\begin{align}\label{gammaDom}
\begin{aligned}
\MoveEqLeft\vert \gamma_{X^{1,m}X^{2,m}}(s\Delta)\vert \\
\leq &\vert \kappa_4 \vert \int_\mathbb{R}\vert \varphi_1(t)\varphi_2(t)\varphi_1(t+s\Delta) \varphi_2(t+s\Delta) \vert \, dt + \sigma^4 \int_\mathbb{R} \vert \varphi_1(t)\varphi_1(t+s\Delta)\vert\, dt \\
&\cdot \int_\mathbb{R} \vert \varphi_2(t)\varphi_2(t+s\Delta)\vert\, dt  +\sigma^4 \int_\mathbb{R} \vert \varphi_1(t)\varphi_2(t+s\Delta)\vert\, dt \cdot \int_\mathbb{R} \vert \varphi_2(t)\varphi_1(t+s\Delta)\vert\, dt.
\end{aligned}
\end{align}
Thus, if we can argue that the three terms on the right-hand side of (\ref{gammaDom}) are summable over $s\in \mathbb{Z}$, we conclude from (\ref{etaTrunc}) that $\eta_m^2 \to \sum_{s\in \mathbb{Z}}\gamma_s = \eta^2$ by dominated convergence. In (\ref{aa}) it was shown that the first term is summable. For the second term we apply H\"{o}lder's inequality to obtain
\begin{align*}
\Bigr\lVert\int_\mathbb{R} \vert \varphi_1(t)\varphi_1(t+\cdot\Delta)\vert\, dt\, \int_\mathbb{R} \vert \varphi_2(t)\varphi_2(t+\cdot\Delta)\vert\, dt\Bigr\rVert_{\ell^1}  \leq \prod_{i=1}^2 \Bigr\lVert \int_\mathbb{R}\vert\varphi_i (t)\varphi_i (t+\cdot \Delta)\vert\, dt \Bigr\rVert_{\ell^{\alpha_i}},
\end{align*}
which is finite by assumption~(\ref{firstAs1}). The last term is handled in the same way using the Cauchy-Schwarz inequality and assumption~(\ref{firstAs3}): 
\begin{align*}
\MoveEqLeft\Bigr\lVert\int_\mathbb{R} \vert \varphi_1(t)\varphi_2(t+\cdot \Delta)\vert\, dt\, \int_\mathbb{R} \vert \varphi_2(t)\varphi_1(t+\cdot \Delta)\vert\, dt\Bigr\rVert_{\ell^1} \\
&\leq \Bigr\lVert\int_\mathbb{R} \vert \varphi_1(t)\varphi_2(t+\cdot \Delta)\vert\, dt\Bigr\rVert_{\ell^2}^2<\infty.
\end{align*} 
Consequently, $Y_m$ converges in distribution to a Gaussian random variable with mean zero and variance $\eta^2$. In light of this, the result is implied by \cite[Proposition~6.3.10]{brockDavis} if the following condition holds:
\begin{align}\label{slutskyCond}
\forall \varepsilon >0:\, \lim_{m \to \infty}\limsup_{n\to \infty}\mathbb{P}\big(\big\vert n^{-1/2}(S_n-\mathbb{E}S_n)- n^{-1/2}(S_n^m-\mathbb{E}S_n^m)\big\vert >\varepsilon \big)  = 0.
\end{align}
In order to show (\ref{slutskyCond}) we find for fixed $m$, using \cite[Theorem~7.1.1]{brockDavis},
\begin{align*}
\MoveEqLeft\limsup_{n\to \infty}\mathbb{E}\big[\big(n^{-1/2}(S_n-\mathbb{E}S_n)- n^{-1/2}(S_n^m-\mathbb{E}S_n^m)\big)^2 \big] \\
&= \limsup_{n\to \infty}n\mathbb{E}\biggr[\biggr(n^{-1}\sum_{s=1}^n \big(X^1_{s\Delta}X^2_{s\Delta} - X^{1,m}_{s\Delta}X^{2,m}_{s\Delta} \big)-\mathbb{E}\big[X^1_0X^2_0-X^{1,m}_0X^{2,m}_0\big]\biggr)^2 \biggr]\\
&= \sum_{s \in \mathbb{Z}}\gamma_{X^1X^2-X^{1,m}X^{2,m}}(s\Delta)
\end{align*}
where $\gamma_{X^1X^2 - X^{1,m}X^{2,m}}$ is the autocovariance function for $(X^1_tX^2_t - X^{1,m}_tX^{2,m}_t)_{t\in \mathbb{R}}$. First, we will establish that $X^{1,m}_0X^{2,m}_0\to X^1_0X^2_0$ in $L^2(\mathbb{P})$. To this end, recall that if a measurable function $f:\mathbb{R}^2\to \mathbb{R}$ is square integrable (with respect to the Lebesgue measure on $\mathbb{R}^2$) and $t \mapsto f(t,t)$ belongs to $L^1\cap L^2$, then the two-dimensional with-diagonal (Stratonovich type) integral $I^{S}(f)$ of $f$ with respect to $(L_t)_{t\in \mathbb{R}}$ is well-defined and by the Hu-Meyer formula,
\begin{align}\label{IntegralBound}
\mathbb{E}\big[I^{S}(f)^2\big] \leq C \Bigr[ \int_{\mathbb{R}^2}f(s,t)^2\, d(s,t) + \int_\mathbb{R} f(t,t)^2\, dt + \Bigr(\int_\mathbb{R}f(t,t)\, dt \Bigr)^2\Bigr]
\end{align} 
for a suitable constant $C>0$. A fundamental property of the Stratonovich integral is that it satisfies the relation
\begin{align*}
I^S(f) = \int_\mathbb{R} g(t)\, dL_t\, \int_\mathbb{R} h(t)\, dL_t,
\end{align*}
when $f (s,t) = g\otimes h (s,t) := g(s)h(t)$ for given measurable functions $g,h:\mathbb{R}\to \mathbb{R}$ such that $g,h,gh\in L^2$. (See \cite{bai2016limit,farre2010multiple} for details.) Since $\varphi_1\varphi_2 \in L^2$ according to (\ref{aa}), we can write $I^S(\varphi_1\otimes \varphi_2(-\cdot) - \varphi_1^m\otimes\varphi_2^m(-\cdot)) = X^1_0X^2_0 - X^{1,m}_0X^{2m}_0$, and hence (\ref{IntegralBound}) shows that
\begin{align}\label{LconvergenceStrat}
\begin{aligned}
\MoveEqLeft\mathbb{E}\big[\big(X^{1}_0X^2_0-X^{1,m}_0X^{2,m}_0\big)^2\big] \\
\leq &\, C\Bigr[ \int_{\mathbb{R}^2}\big(\varphi_1(s)\varphi_2(t)-\varphi^m_1(s)\varphi^m_2(t)\big)^2\, d(s,t)+ \int_\mathbb{R} \big(\varphi_1(t)\varphi_2(t)-\varphi^m_1(t)\varphi^m_2(t)\big)^2\, dt \\
&+ \Bigr(\int_\mathbb{R}\varphi_1(t)\varphi_2(t)\, dt-\int_\mathbb{R}\varphi^m_1(t)\varphi^m_2(t)\, dt \Bigr)^2\Bigr]
\end{aligned}
\end{align}
for a suitable constant $C>0$. It is clear that the three terms on the right-hand side of (\ref{LconvergenceStrat}) tend to zero as $m$ tends to infinity, and thus we have that $X^{1,m}_tX^{2,m}_t\to X^{1}_tX^{2}_t$ in $L^2 (\mathbb{P})$. In particular, this shows that $\gamma_{X^1X^2-X^{1,m}X^{2,m}}(s\Delta) \to 0$ as $m\to \infty$ for each $s \in \mathbb{Z}$. By using the same type of bound as in (\ref{gammaDom}), we establish the existence of a function $h:\mathbb{Z}\to [0,\infty)$ in $\ell^1$ with $\vert \gamma_{X^1X^2-X^{1,m}X^{2,m}}(s\Delta)\vert \leq h(s)$ for all $s\in \mathbb{Z}$ and, consequently,
\begin{align*}
\MoveEqLeft\lim_{m \to \infty}\limsup_{n\to \infty} \mathbb{E}\big[\big(n^{-1/2}(S_n-\mathbb{E}S_n)- n^{-1/2}(S_n^m-\mathbb{E}S_n^m)\big)^2 \big] \\&= \lim_{m \to \infty}\sum_{s \in \mathbb{Z}}\gamma_{X^1X^2-X^{1,m}X^{2,m}}(s\Delta)  = 0
\end{align*}
according to Lebesgue's theorem. In light of (\ref{slutskyCond}), we have finished the proof. 
\end{proof}
Relying on the ideas of Young's convolution inequality, we obtain the following lemma:
\begin{lemma}\label{youngConvolution}
Let $\alpha,\beta,\gamma \in [1,\infty]$ satisfy $1/\alpha + 1/\beta-1 = 1/\gamma$. Suppose that 
\begin{align*}
\big(t \mapsto \big\lVert f(t+\cdot \Delta) \big\rVert_{\ell^\alpha}\big)\in L^{2\alpha}([0,\Delta])\quad \text{and} \quad \big(t\mapsto \big\lVert g(t+\cdot \Delta) \big\rVert_{\ell^\beta}\big)\in L^{2\beta}([0,\Delta]).
\end{align*}
Then it holds that $\int_\mathbb{R}\vert f(t)g(t+\cdot\Delta)\vert\, dt \in \ell^\gamma$.
\end{lemma}
\begin{proof}
First observe that, for any measurable function $h:\mathbb{R}\to \mathbb{R}$ and $p\in [1,\infty]$, $h\in L^p$ if and only if $t\mapsto \big\lVert h(t + \cdot\Delta)\big\rVert_{\ell^p}$ belongs to $L^p([0,\Delta])$. In particular, this ensures that $f\in L^\alpha$ and $g \in L^\beta$. If $\gamma=\infty$ then $1/\alpha+1/\beta  = 1$, and the result follows immediately from H\"{o}lder's inequality. Hence, we will restrict the attention to $\gamma< \infty$, in which case we necessarily also have that $\alpha,\beta<\infty$. First, consider the case where $\alpha,\beta\neq \gamma$, or equivalently $\alpha,\beta,\gamma>1$, and set $\alpha' = \alpha/(\alpha-1)$ and $\beta' = \beta/(\beta-1)$. Note that these definitions ensure that $\alpha' (1-\beta/\gamma) = \beta$, $\beta'(1-\alpha/\gamma)=\alpha$ and $1/\alpha' + 1/\beta'+1/\gamma = 1$. Hence, using the H\"{o}lder inequality and the facts that $f \in L^\alpha$ and $g \in L^\beta$,
\begin{align*}
\int_\mathbb{R}\vert f(t)g(t+s\Delta)\vert\, dt \leq &\, \Bigr( \int_\mathbb{R}\vert f(t)\vert^\alpha \vert g(t+s\Delta)\vert^\beta\, dt\Bigr)^{1/\gamma}\cdot \Bigr(\int_\mathbb{R}\vert f(t)\vert^{\beta' (1-\alpha/\gamma)}\, dt \Bigr)^{1/\beta'}\\
&\cdot \Bigr(\int_\mathbb{R} \vert g (t+s\Delta)\vert^{\alpha'(1-\beta/\gamma)}\, dt \Bigr)^{1/\alpha'}\\
=&\, M^{1/\gamma} \Bigr(\int_\mathbb{R}\vert f(t)\vert^\alpha \vert g(t+s\Delta)\vert^\beta\, dt\Bigr)^{1/\gamma}
\end{align*}
for a suitable constant $M< \infty$. By raising both sides to the $\gamma$-th power, summing over $s\in \mathbb{Z}$ and applying the Cauchy-Schwarz inequality we obtain that
\begin{align}\label{summingIneq}
\begin{aligned}
\Bigr\lVert\int_\mathbb{R}\vert f(t)g(t+\cdot\Delta)\vert\, dt \Bigr\rVert^\gamma_{\ell^\gamma} &\leq M \int_\mathbb{R} \vert f(t)\vert^\alpha  \lVert g(t+\cdot\Delta)\rVert_{\ell^\beta}^\beta\, dt\\
&\leq M \Bigr(\int_0^\Delta\lVert f(t+\cdot\Delta)\rVert^{2\alpha}_{\ell^\alpha}\, dt\Bigr)^{1/2}\Bigr(\int_0^\Delta\lVert g(t+\cdot\Delta)\rVert_{\ell^\beta}^{2\beta}\, dt\Bigr)^{1/2},
\end{aligned}
\end{align}
which is finite, and thus we have finished the proof in case $\alpha,\beta\neq \gamma$. If, e.g., $\gamma=\alpha\neq \beta$ then $\alpha>1$. Again, set $\alpha'=\alpha/(\alpha-1)$ and note that $1/\alpha' +1/\gamma =1$, so the H\"{o}lder inequality ensures that
\begin{align*}
\int_\mathbb{R}\vert f(t)g(t+s\Delta)\vert\, dt&\leq \Bigr(\int_\mathbb{R}\vert f(t)\vert^\alpha \vert g(t+s\Delta)\vert^\beta\, dt \Bigr)^{1/\gamma}\cdot \Bigr(\int_\mathbb{R}\vert g(t)\vert^\beta\, dt \Bigr)^{1/\alpha'},
\end{align*}
and hence the inequalities in (\ref{summingIneq}) hold in this case as well for a suitable constant $M>0$. Finally if $\alpha=\beta=\gamma=1$, we compute that
\begin{align*}
\Bigr\lVert \int_\mathbb{R}\vert f(t)g(t+\cdot \Delta)\vert\, dt \Bigr\rVert_{\ell^1} &=\int_0^\Delta \lVert f(t+\cdot \Delta) \rVert_{\ell^1}\lVert g(t+\cdot \Delta) \rVert_{\ell^1}\, dt\\
&\leq (\int_0^\Delta\lVert f(t+\cdot\Delta)\rVert^{2}_{\ell^1}\, dt\Bigr)^{1/2}\Bigr(\int_0^\Delta\lVert g(t+\cdot\Delta)\rVert_{\ell^1}^{2}\, dt\Bigr)^{1/2}<\infty,
\end{align*}
which finishes the proof.
\end{proof}

\begin{proof}[Proof of Theorem~\ref{theoremIntro2}]
To show that statement (\ref{introAs12}) implies the stated weak convergence of $(S_n-\mathbb{E}S_n)/\sqrt{n}$, it suffices to check that assumptions~(\ref{firstAs1})-(\ref{firstAs2}) of Theorem~\ref{Qform} are satisfied. Initially note that, in view of the observation in the beginning of the proof of Lemma~\ref{youngConvolution}, the imposed assumptions imply that
\begin{align*}
\varphi_i \in L^\beta\quad \text{and} \quad \big( t \mapsto \lVert\varphi_i (t+\cdot\Delta)\rVert_{\ell^\beta}\big) \in L^{2\beta}([0,\Delta])
\end{align*}
for all $\beta \in [\alpha_i,2]$. Since
\begin{align*}
\tfrac{1}{2}\in \big\{\tfrac{1}{\beta_1} + \tfrac{1}{\beta_2}-1\, :\, \alpha_i\leq \beta_i \leq 2 \big\},
\end{align*}
we can thus assume that $\alpha_1,\alpha_2 \in [1,2]$ are given such that $1/\alpha_1+1/\alpha_2 -1= 1/2$. Next, define $\gamma_i$ by the relation $1/\gamma_i = 2/\alpha_i-1$ if $\alpha_i<2$ and $\gamma_i = \infty$ if $\alpha_i = 2$. In this case, $1/\gamma_1 + 1/\gamma_2 = 1$. By applying Lemma~\ref{youngConvolution} with $f =g= \varphi_i$, $\alpha=\beta=\alpha_i$ and $\gamma=\gamma_i$, we deduce that (\ref{firstAs1}) of Theorem~\ref{Qform} holds. Assumption~(\ref{firstAs3}) of Theorem~\ref{Qform} holds as well by Lemma~\ref{youngConvolution} with $f  = \varphi_1$, $g = \varphi_2$, $\alpha=\alpha_1$, $\beta = \alpha_2$ and $\gamma=2$. Finally, we have that assumption~(\ref{firstAs2}) of Theorem~\ref{Qform} is satisfied, since
\begin{align*}
\MoveEqLeft\int_0^\Delta\lVert \varphi_1 (t+\cdot\Delta) \varphi_2 (t+\cdot\Delta)\rVert_{\ell^1}^2\, dt\\
&\leq \Bigr(\int_0^\Delta \lVert\varphi_1 (t+\cdot\Delta)\rVert_{\ell^2}^4\, dt \Bigr)^{1/2}\Bigr(\int_0^\Delta \lVert\varphi_2 (t+\cdot\Delta)\rVert_{\ell^2}^4\, dt \Bigr)^{1/2}<\infty,
\end{align*}
where we have applied the Cauchy-Schwarz inequality both for sums and integrals.

The last part of the proof (concerning statement (\ref{introAs22}) in the theorem) amounts to show that if $\varphi_1,\varphi_2\in L^4$ and $\alpha_1,\alpha_2 \in (1/2,1)$ are given such that $\alpha_1 + \alpha_2 > 3/2$ and
\begin{align}\label{functionDecay}
c_i:=\sup_{t\in \mathbb{R}} \vert t \vert^{\alpha_i}\vert \varphi_i (t)\vert <\infty,\quad i=1,2,
\end{align}
then $t\mapsto \lVert \varphi_i (t+\cdot \Delta) \rVert_{\ell^\kappa}^{\kappa}$ belongs to $L^2([0,\Delta])$ for $\kappa \in \{\beta_i,2\}$ where $\beta_i\in (1/\alpha_i,2]$, $i=1,2$, are given such that $1/\beta_1 + 1/\beta_2 \geq 3/2$. To show this, consider $\kappa \in \{\beta_i,2\}$ and write
\begin{align}\label{sumDec}
\begin{aligned}
\lVert \varphi_i (t+\cdot \Delta)\rVert_{\ell^\kappa}^\kappa =&\, \vert \varphi_i(t+\Delta)\vert^{\kappa}+ \vert \varphi_i(t)\vert^{\kappa} +\vert \varphi_i(t-\Delta)\vert^{\kappa} \\
& + \sum_{s=2}^\infty \vert \varphi_i (t+s\Delta)\vert^\kappa + \sum_{s=2}^{\infty} \vert \varphi_i (t-s\Delta)\vert^\kappa
\end{aligned}
\end{align}
for $t\in [0,\Delta]$. Since $\varphi_i\in L^4$, the first three terms on the right-hand side of (\ref{sumDec}) belong to $L^2([0,\Delta])$. The last two terms belong to $L^\infty ([0,\Delta])$, since
\begin{align*}
\sup_{t\in [0,\Delta]}\sum_{s=2}^\infty \vert \varphi_i (t\pm s\Delta)\vert^\kappa \leq c_i^\kappa \Delta^{-\kappa\alpha_i}\sum_{s=1}^\infty s^{-\kappa\alpha_i}<\infty
\end{align*}
by (\ref{functionDecay}), and hence $\big(t\mapsto \lVert \varphi_i(t+\cdot \Delta) \rVert_{\ell^\kappa}^\kappa\big)\in L^2([0,\Delta])$.
\end{proof}

\begin{proof}[Proof of Theorem~\ref{quadraticForms}]
Initially, we note that
\begin{align}\label{QnDecompose}
Q_n = \sum_{t=1}^n X_{t\Delta}\int_\mathbb{R}\sum_{s=t-n}^{t-1}b(s)\varphi ((t-s)\Delta -u)\, dL_u = S_n - \varepsilon_n -\delta_n,
\end{align}
where
\begin{align*}
S_n &= \sum_{t = 1}^n X_{t\Delta}\int_\mathbb{R}b\star \varphi (t\Delta -u)\, dL_u,\\
\varepsilon_n &= \sum_{t = 1}^n X_{t\Delta}\int_\mathbb{R}\sum_{s =t}^{\infty} b(s)\varphi ((t-s)\Delta-u)\, dL_u\quad \text{and} \\
\delta_n &= \sum_{t = 1}^n X_{t\Delta}\int_\mathbb{R}\sum_{s=-\infty}^{t-n-1}b(s)\varphi ((t-s)\Delta -u)\, dL_u.
\end{align*}
As pointed out in Remark~\ref{QnApprox}, the imposed assumptions ensure that Theorem~\ref{Qform} is applicable with $\varphi_1 = \varphi$ and $\varphi_2 = \vert b\vert \star \vert\varphi\vert$ (in particular, when $\varphi_2 = b\star \varphi$), and thus $(S_n-\mathbb{E}S_n)/\sqrt{n}\overset{d}{\to}N(0,\eta^2)$ where $\eta^2$ is given by (\ref{SnVariance}). By using that $b$ is even we compute
\begin{align*}
\MoveEqLeft\sigma^2\sum_{s\in \mathbb{Z}}\gamma_X(s\Delta)\int_\mathbb{R} (b\star \varphi) (t) (b\star \varphi)(t+s\Delta)\, dt\\
&= \sum_{s\in \mathbb{Z}} \sum_{u,v\in \mathbb{Z}} b(u)b(v)\gamma_X((s+u)\Delta)\gamma_X((s+v)\Delta) = \lVert (b \star \gamma_X)(\cdot \Delta) \rVert_{\ell^2}^2
\end{align*}
and
\begin{align*}
\MoveEqLeft \sigma^4\sum_{s\in \mathbb{Z}} \int_\mathbb{R} \varphi (t) (b\star \varphi) (t+s\Delta)\, dt\cdot \int_\mathbb{R} (b\star \varphi)(t)\varphi (t+s\Delta)\, dt\\
&= \sum_{s\in \mathbb{Z}}\sum_{u,v\in \mathbb{Z}}b(u)b(v) \gamma_X ((s-u)\Delta)\gamma_X ((s+v)\Delta) = \lVert (b\star \gamma_X) (\cdot \Delta) \rVert_{\ell^2}^2,
\end{align*}
it follows that $\eta^2$ coincides with (\ref{quadraticFormVariance}).
%\textit{Part \ref{conv1}:} Since $S_n$ is of the same form as in (\ref{relevantQuantity}) with $\varphi_1(t) = \varphi (t)$ and $\varphi_2 (t) = \sum_{s\in \mathbb{Z}}b(s) \varphi (t-s\Delta)$, it suffices to argue that $\varphi_2 \in L^2\cap L^4$ and that assumptions~(\ref{firstAs1})-(\ref{firstAs2}) of Theorem~\ref{Qform} are satisfied. However, it is straightforward to check that imposing assumptions~(\ref{as1})-(\ref{as4}) is equivalent to assuming that $\psi = \sum_{s \in \mathbb{Z}}\vert b(s) \varphi (\cdot - s\Delta)\vert$ belongs to $L^2 \cap L^4$ and that (\ref{firstAs1})-(\ref{firstAs2}) of Theorem~\ref{Qform} hold for $\varphi$ and $\psi$. Consequently, it follows from Theorem~\ref{Qform} that $(S_n - \mathbb{E}S_n)/\sqrt{n}$ converges in distribution to a Gaussian %random variable with mean zero and variance
%\begin{align*}
%\MoveEqLeft\kappa_4 \sum_{t,s,u \in \mathbb{Z}} b(s)b(u)\int_\mathbb{R}\varphi (v)\varphi (v-s\Delta)\varphi (t\Delta+v) \varphi ((t-u)\Delta +v)\, dv\\
%&+ \sum_{t\in \mathbb{Z}}\gamma (t\Delta)\sum_{s,u\in \mathbb{Z}}b(s)b(u)\gamma ((t+s-u)\Delta)\\
%&+ \sum_{t\in \mathbb{Z}}\Bigr(\sum_{s\in \mathbb{Z}}b(s)\gamma ((t-s)\Delta) \Bigr)\Bigr(\sum_{s\in \mathbb{Z}}b(s)\gamma ((t+s)\Delta) \Bigr),
%\end{align*}
%which is equal to $\eta^2$ given in (\ref{quadraticFormVariance}) since $b$ is even.
In light of the decomposition (\ref{QnDecompose}) and Slutsky's theorem, we have shown the result if we can argue that $\text{Var}(\varepsilon_n)/n \to 0$ and $\text{Var}(\delta_n)/n \to 0$ as $n\to \infty$. We will only show that $\text{Var}(\varepsilon_n)/n \to 0$, since arguments verifying that $\text{Var}(\delta_n)/n\to 0$ are similar. Define $a (t)= \int_\mathbb{R}\varphi (s)\varphi (t\Delta +s)\, ds$ and note that we have the identities
\begin{align*}
\mathbb{E}\varepsilon_n &= \sigma^2\sum_{t=1}^n \sum_{s=-\infty}^0 a(t-s)b(t-s)\quad \text{and} \\
\mathbb{E}\varepsilon_n^2 &= \sum_{t,s=1}^n \sum_{u =t}^\infty \sum_{v =s}^\infty b(u)b(v) \mathbb{E}[X_{t\Delta}X_{s\Delta}X_{(t-u)\Delta} X_{(s-v)\Delta}]\\
&=\sum_{t,s,u,v\in \mathbb{Z}} b(t-u)b(s-v) \mathbb{E}[X_{t\Delta}X_{s\Delta}X_{u\Delta}X_{v\Delta}] \mathds{1}_{\{1\leq t,s\leq n\}}\mathds{1}_{\{u,v\leq 0\}}.
\end{align*}
Moreover, with
\begin{align*}
c(t,s,u) =\int_\mathbb{R}\varphi (t\Delta + v)\varphi (s\Delta + v)\varphi (u\Delta + v)\varphi (v)\, dv,
\end{align*}
it follows by Lemma~\ref{fourthMoment} that
\begin{align*}
\mathbb{E}[X_{t\Delta}X_{s\Delta}X_{u\Delta}X_{v\Delta}] =&\, \kappa_4 c(t-v,s-v,u-v) + \sigma^4 a(t-s)a(u-v) \\
&+ \sigma^4 a(t-u)a(s-v) + \sigma^4 a(t-v)a(s-u)
\end{align*}
for any $t,s,u,v\in \mathbb{Z}$. Thus, we establish the identity
\begin{align}\label{keyEquation}
\begin{aligned}
n^{-1}\text{Var}(\varepsilon_n) =&\, \kappa_4 n^{-1} \sum_{t,s,u,v\in \mathbb{Z}} b(t-u)b(s-v)c(t-v,s-v,u-v)\mathds{1}_{\{1\leq t,s\leq n\}}\mathds{1}_{\{u,v\leq 0\}}\\
&+\sigma^4n^{-1}\sum_{t,s,u,v\in \mathbb{Z}} a(t-s)a(u-v)b(t-u)b(s-v)\mathds{1}_{\{1\leq t,s\leq n\}}\mathds{1}_{\{u,v\leq 0\}} \\
&+\sigma^4n^{-1}\sum_{t,s,u,v\in \mathbb{Z}} a(t-v)a(s-u)b(t-u)b(s-v)\mathds{1}_{\{1\leq t,s\leq n\}}\mathds{1}_{\{u,v\leq 0\}}.
\end{aligned}
\end{align}
Thus, it suffices to argue that each of the three terms on the right-hand side of (\ref{keyEquation}) tends to zero as $n$ tends to infinity. Regarding the first term, by a change of variables from $(t,s,u,v)$ to $(t-v,s-u,u-v,v)$, we have
\begin{align}\label{term1}
\begin{aligned}
\MoveEqLeft n^{-1}\sum_{t,s,u,v\in \mathbb{Z}} b(t-u)b(s-v)c(t-v,s-v,u-v)\mathds{1}_{\{1\leq t,s\leq n\}}\mathds{1}_{\{u,v\leq 0\}}\\
&= \sum_{t,s,u\in \mathbb{Z}}b(t-u)b(s+u)c(t,s+u,u) n^{-1}\sum_{v \in \mathbb{Z}}\mathds{1}_{\{1\leq t+v,s+u+v\leq n\}}\mathds{1}_{\{u+v,v\leq 0\}}.
\end{aligned}
\end{align}
Since for fixed $t,s,u\in \mathbb{Z}$,
\begin{align*}
\sum_{v \in \mathbb{Z}}\mathds{1}_{\{1\leq t+v,s+u+v\leq n\}}\mathds{1}_{\{u+v,v\leq 0\}} \leq \min\{\vert t \vert,n\},
\end{align*}
it will follow that the expression in (\ref{term1}) tends to zero as $n$ tends to infinity by Lebesgue's theorem on dominated convergence if
\begin{align}\label{integrability1}
\sum_{t,s,u \in \mathbb{Z}}\vert b(t)b(s)c(t+u,s,u)\vert < \infty.
\end{align}
To show (\ref{integrability1}) we use that the function $t\mapsto \lVert \varphi (t+\cdot \Delta) (\vert b \vert \star \vert \varphi \vert)(t+\cdot \Delta) \rVert_{\ell^1}$ belongs to $L^2([0,\Delta])$ (by assumption~(\ref{as3})) and is periodic with period $\Delta$:
\begin{align*}
\MoveEqLeft\sum_{t,s,u \in \mathbb{Z}} \vert b(t) b(s)c(t+u,s,u)\vert \\
&\leq \sum_{u\in \mathbb{Z}}\int_\mathbb{R}\vert \varphi (v)\vert\, (\vert b\vert\star \vert \varphi \vert)(v)\, \vert \varphi (v+u\Delta)\vert\, (\vert b\vert\star \vert \varphi \vert)(v+u\Delta)\, dv
\\
&= \int_0^\Delta \lVert \varphi (v+\cdot \Delta) (\vert b \vert \star \vert \varphi \vert)(v+\cdot \Delta) \rVert_{\ell^1}^2 \, dv<\infty.
\end{align*}
Hence, (\ref{term1}) tends to zero. We will handle the second term on the right-hand side of (\ref{keyEquation}) in a similar way. In particular, by a change of variables from $(t,s,u,v)$ to $(t,t-s,s-u,t-v)$,
\begin{align}\label{term2}
\begin{aligned}
\MoveEqLeft n^{-1}\sum_{t,s,u,v\in \mathbb{Z}} a(t-s)a(u-v)b(t-u)b(s-v)\mathds{1}_{\{1\leq t,s\leq n\}}\mathds{1}_{\{u,v\leq 0\}}\\
&= \sum_{s,u,v\in \mathbb{Z}}a(s) a(v-u-s) b(s+u)b(v-s) n^{-1}\sum_{t\in \mathbb{Z}} \mathds{1}_{\{1\leq t, t-s\leq n\}} \mathds{1}_{\{t-s-u,t-v\leq 0\}}.
\end{aligned}
\end{align}
For fixed $s,u,v \in \mathbb{Z}$,
\begin{align*}
\sum_{t\in \mathbb{Z}} \mathds{1}_{\{1\leq t,t-s\leq n\}} \mathds{1}_{\{t-s-u,t-v\leq 0\}} \leq \min \{\vert v \vert,n\},
\end{align*}
and since
\begin{align}\label{dominatedConv23}
\begin{aligned}
\MoveEqLeft\sum_{s,u,v\in \mathbb{Z}} \vert a(s) a(v-u-s) b(s+u)b(v-s)\vert\\
&\leq \lVert a\rVert_{\ell^\alpha} \Bigr\lVert\sum_{u,v\in \mathbb{Z}}\vert a(v-u-\cdot) b(\cdot+u)b(v-\cdot)\vert \Bigr\rVert_{\ell^\beta}
\\
&\leq \Bigr\lVert \int_\mathbb{R}\vert \varphi (u)\varphi (u+\cdot \Delta)\vert\, du\Bigr\rVert_{\ell^\alpha}  \Bigr\lVert \int_{\mathbb{R}}(\vert b \vert \star \vert \varphi \vert)(u)\, (\vert b \vert \star \vert \varphi \vert)(u+\cdot \Delta)\, du \Bigr\rVert_{\ell^\beta}<\infty
\end{aligned}
\end{align}
by assumption~(\ref{as1}), it follows again by dominated convergence that (\ref{term2}) tends to zero as $n$ tends to infinity. Finally, for the third term on the right-hand side of (\ref{keyEquation}), we make a change of variables from $(t,s,u,v)$ to $(t-u,s-t,u-v,v)$ and establish the inequality
\begin{align}\label{term3}
\begin{aligned}
\MoveEqLeft n^{-1}\sum_{t,s,u,v\in \mathbb{Z}} a(t-v)a(s-u)b(t-u)b(s-v)\mathds{1}_{\{1\leq t,s\leq n\}}\mathds{1}_{\{u,v\leq 0\}}\\
&\leq \sum_{t,s,u\in \mathbb{Z}} a(t+u)a(t+s)b(t)b(t+s+u)n^{-1}\min \{\vert t+u\vert,n\}.
\end{aligned}
\end{align}
The right-hand side of (\ref{term3}) tends to zero as $n$ tends to infinity by dominated convergence using (\ref{dominatedConv23}) and that $a$ is even. Consequently, (\ref{keyEquation}) shows that $\text{Var}(\varepsilon_n)/n \to 0$ as $n \to \infty$, which ends the proof.
\end{proof}

\begin{proof}[Proof of Theorem~\ref{theoremIntro}]
To show (\ref{introAs1}), define $\gamma\in [1,2]$ by the relation $1/\gamma=1/\alpha + 1/\beta -1$ and note that $1/\alpha+ 1/\gamma \geq 3/2$. According to Remark~\ref{QnApprox} it suffices to check that the assumptions of Theorem~\ref{Qform} are satisfied for the functions $\varphi$ and $\vert b \vert \star \vert \varphi \vert$, which in turn follows from the same arguments as in the proof of Theorem~\ref{theoremIntro2} if
\begin{align}
\big(t\mapsto \lVert  \varphi (t+\cdot \Delta) \rVert_{\ell^\alpha}^\alpha + \lVert  \varphi (t+\cdot \Delta) \rVert_{\ell^2}^2 \big)&\in L^2([0,\Delta])\quad \text{and}\label{QnCorTarget}\\
\big(t\mapsto \lVert  (\vert b\vert \star \vert \varphi \vert) (t+\cdot \Delta) \rVert_{\ell^\gamma}^\gamma + \lVert  (\vert b\vert \star \vert \varphi \vert) (t+\cdot \Delta) \rVert_{\ell^2}^2 \big)&\in L^2([0,\Delta]).\label{QnCorTarget2}
\end{align}
Condition (\ref{QnCorTarget}) holds by assumption (since $\alpha \leq 2$), so we only need to prove (\ref{QnCorTarget2}). If $\beta =1$ so that $b$ is summable, it follows from Jensen's inequality that
\begin{align*}
\big(\vert b \vert \star \vert \varphi\vert \big)(t)^\kappa \leq \lVert b \rVert_{\ell^1}^{\kappa-1} \sum_{s\in \mathbb{Z}}\vert b (s)\vert\, \vert\varphi(t+s\Delta)\vert^\kappa,
\end{align*}
and thus $\lVert (\vert b \vert \star \vert \varphi \vert)(t+\cdot \Delta) \rVert_{\ell^\kappa}^\kappa\leq \lVert b \rVert_{\ell^1}^\kappa \lVert\varphi (t+\cdot \Delta) \rVert_{\ell^\kappa}^\kappa$ for any $\kappa \geq 1$. Since $\alpha = \gamma$ when $\beta =1$, this shows that (\ref{QnCorTarget}) implies (\ref{QnCorTarget2}). Next if $\beta >1$, set $\beta'=\beta /(\beta-1)$. As in the proof of Lemma~\ref{youngConvolution} (replacing integrals by sums), we can use the H\"{o}lder inequality to obtain the estimate
\begin{align*}
(\vert b \vert \star \vert \varphi \vert)(t)\leq M^{1/\gamma}
\Bigr(\sum_{s\in \mathbb{Z}} \vert \varphi (t+s\Delta)\vert^\alpha \Bigr)^{1/\beta'}\Bigr(\sum_{s\in \mathbb{Z}}\vert b(s)\vert^\beta \vert \varphi (t+s\Delta)\vert^\alpha \Bigr)^{1/\gamma}
\end{align*}
for some constant $M>0$. By raising both sides to the $\gamma$-th power and exploiting the periodicity of $t\mapsto \lVert \varphi (t+\cdot \Delta) \rVert_{\ell^\alpha}^\alpha$, it follows that
\begin{align}\label{keySufficientRelation}
\begin{aligned}
\lVert (\vert b \vert \star \vert \varphi \vert)(t+\cdot \Delta) \rVert_{\ell^\gamma}^\gamma &\leq M\Bigr(\sum_{s\in \mathbb{Z}}\vert \varphi (t+s\Delta)\vert^\alpha  \Bigr)^{\gamma/\beta'}\sum_{s\in \mathbb{Z}}\vert b(s) \vert^\beta \sum_{u\in \mathbb{Z}}\vert\varphi (t+(s+u)\Delta)\vert^\alpha\\
&= M\lVert b \rVert_{\ell^\beta}^\beta \lVert\varphi (t+\cdot\Delta)\rVert^{\gamma}_{\ell^\alpha}
\end{aligned}
\end{align}
for a sufficiently large constant $M>0$. Since $\gamma \leq 2$, (\ref{keySufficientRelation}) and the assumption $\big(t\mapsto \lVert\varphi (t+\cdot \Delta) \rVert_{\ell^\alpha}^\alpha \big)\in L^{4/\alpha}([0,\Delta])$ show that $\big(t\mapsto \lVert (\vert b \vert \star \vert \varphi \vert)(t+\cdot \Delta) \rVert_{\ell^\gamma}^\gamma \big) \in L^2([0,\Delta])$. To show $t\mapsto \lVert (\vert b \vert \star \vert \varphi \vert)(t+\cdot \Delta) \rVert_{\ell^2}^2\in L^2([0,\Delta])$, we note that the assumption $2/\alpha + 1/\beta \geq 5/2$ ensures that we may choose $\beta^* \in [\beta,2]$ such that $1/\alpha + 1/\beta^* = 3/2$. Using the same type of arguments as above, now with $\alpha$, $\beta^*$ and $\gamma^*=2$ instead of $\alpha$, $\beta$ and $\gamma$, we obtain the inequality
\begin{align*}
\lVert (\vert b \vert \star \vert \varphi \vert)(t+\cdot \Delta) \rVert_{\ell^2}^2 \leq M\lVert b \rVert_{\ell^{\beta^*}}^{\beta^*} \lVert\varphi (t+\cdot\Delta)\rVert^{2}_{\ell^{\alpha}}.
\end{align*}
Due to the fact that $(t\mapsto \rVert \varphi (t+ \cdot \Delta) \lVert_{\ell^{\alpha}}^{\alpha})\in L^{4/\alpha}([0,\Delta])$, this shows that $\big(t\mapsto \lVert (\vert b \vert \star \vert \varphi \vert)(t+\cdot \Delta) \rVert_{\ell^2}^2\big)\in L^2([0,\Delta])$ and, thus, ends the proof under statement (\ref{introAs1}).

In view of the above, to show the last part of the theorem (concerning statement (\ref{introAs2})), it suffices to argue that if $\varphi \in L^4$,
\begin{align*}
c_1 := \sup_{t\in \mathbb{R}} \vert t \vert^{1-\alpha/2}  \vert \varphi (t)\vert<\infty \quad \text{and}\quad c_2 := \sup_{t\in \mathbb{Z}}\vert t \vert^{1-\beta} \vert b (t)\vert<\infty
\end{align*}
for some $\alpha,\beta >0$ with $\alpha + \beta <1/2$, then there exist $p,q\in [1,2]$ such that $2/p + 1/q\geq 5/2$, $b\in \ell^q$ and $\big(t\mapsto \lVert \varphi (t+\cdot \Delta) \rVert_{\ell^\kappa}\big)\in L^4([0,\Delta])$ for $\kappa \in \{p,2\}$. To do so observe that
\begin{align*}
\tfrac{5}{2}\in \big\{\tfrac{2}{p}+\tfrac{1}{q}\, :\, \tfrac{2}{2-\alpha}<p\leq 2,\, \tfrac{1}{1-\beta}<q\leq 2\big\},
\end{align*}
and hence we may (and do) fix $p,q \in [1,2]$ such that $2/p + 1/q \geq 5/2$, $p (\alpha/2-1)<-1$ and $q(\beta -1)<-1$. With this choice it holds that $b\in \ell^q$, since
\begin{align*}
\lVert b \rVert_{\ell^q}^q \leq \vert b (0) \vert^q + 2c_2^q \sum_{s=1}^\infty s^{q(\beta-1)} <\infty.
\end{align*}
We can use the same type of arguments as in the last part of the proof of Theorem~\ref{theoremIntro2} to conclude that $\big(t\mapsto \lVert \varphi (t+\cdot \Delta) \rVert_{\ell^\kappa}^\kappa\big)\in L^{4/\kappa}([0,\Delta])$ for $\kappa \in \{p,2\}$. Indeed, in view of the decomposition (\ref{sumDec}) (with $\varphi$ playing the role of $\varphi_i$) and the fact that $\varphi \in L^4$, it suffices to argue that $\sup_{t\in [0,\Delta]}\sum_{s=2}^\infty \vert \varphi (t\pm s\Delta)\vert^\kappa<\infty$. However, this is clearly the case as $\kappa (\alpha/2-1)\leq p(\alpha/2-1)<-1$ and, thus,
\begin{align*}
\sup_{t\in [0,\Delta]}\sum_{s=2}^\infty \vert \varphi (t+s\Delta)\vert^\kappa \leq c_1^\kappa \Delta^{\kappa (\alpha/2-1)}\sum_{s=1}^\infty s^{\kappa (\alpha/2-1)}<\infty.
\end{align*}
This ends the proof of the result.
\end{proof}
\subsection*{Acknowledgments}
The research was supported by Danish Council for Independent Research grant DFF-4002-00003. 

\bibliographystyle{chicago}

\end{document}